\documentclass[sn-basic]{sn-jnl}

\usepackage{graphicx}%
\usepackage{multirow}%
\usepackage{amsmath,amssymb,amsfonts}%
\usepackage{amsthm}%
\usepackage{mathrsfs}%
\usepackage[title]{appendix}%
\usepackage{xcolor}%
\usepackage{textcomp}%
\usepackage{manyfoot}%
\usepackage{booktabs}%
\usepackage{algorithm}%
\usepackage{algorithmicx}%
\usepackage{algpseudocode}%
\usepackage{listings}%
\usepackage[capitalize]{cleveref}
\usepackage{braket}
\usepackage{nicefrac}

\usepackage{stackengine}
\usepackage{braket}
\usepackage{amsmath}
\usepackage{amscd}
\usepackage{amsfonts}
\usepackage{amssymb}
\usepackage{subfigure}
\usepackage{tabularray}
\UseTblrLibrary{booktabs}
\usepackage{hyperref} 
\usepackage{pgfplots}
\pgfplotsset{compat=newest}
\usetikzlibrary{calc}
\usepackage{tikzscale}
\usepackage{tikz-3dplot}
\usepackage[capitalize]{cleveref}
\crefname{equation}{}{}
\usepackage{nicefrac}

\theoremstyle{thmstyleone}%
\newtheorem{theorem}{Theorem}
\newtheorem{proposition}[theorem]{Proposition}%
\newtheorem{corollary}[theorem]{Corollary}%

\theoremstyle{thmstyletwo}%
\newtheorem{lemma}{Lemma}%
\newtheorem{remark}{Remark}%

\theoremstyle{thmstylethree}%

\usepackage[varvw,slantedGreek]{newtxmath}
\newcommand{\N}{\varmathbb{N}}

\newcommand{\R}{\varmathbb{R}}

\newcommand{\T}{\mathcal{T}}
\newcommand{\abs}[1]{\lvert#1\rvert}

\newcommand{\norm}[1]{\lVert#1\rVert}

\newcommand{\ltwonorm}[1]{\norm{#1}_{L^2(\Omega)}}
\newcommand{\ltwonormdo}[1]{\norm{#1}_{L^2(\partial \Omega)}}

\newcommand{\htwonorm}[1]{\norm{#1}_{H^2(\Omega)}}

\newcommand{\lpnorm}[1]{\norm{#1}_{L^p(\Omega)}}

\newcommand{\ltwonormk}[1]{\norm{#1}_{L^2(K)}}

\newcommand{\linfnorm}[1]{\norm{#1}_{L^\infty(\Omega)}}

\newcommand{\into}{\int_\Omega}
\newcommand{\intdo}{\int_{\partial \Omega}}

\newcommand{\half}{\frac{1}{2}}
\newcommand{\nhalf}{\nicefrac{1}{2}}

\newcommand{\dist}{\operatorname{dist}}

\newcommand{\Qad}{{Q_\mathrm{ad}}}
\newcommand{\Qadh}{{Q_{\mathrm{ad},h}}}

\newcommand{\oq}{\bar q}
\newcommand{\ou}{\bar u}
\newcommand{\oz}{\bar z}

\newcommand{\dq}{\delta\mspace{-2mu}q}
\newcommand{\du}{\delta\mspace{-2mu}u}

\newcommand{\lh}{\abs{\ln h}}
\newcommand{\Lap}{\upDelta}
\renewcommand{\phi}{\varphi}
\newcommand{\Om}{\Omega}

\newcommand{\logLogSlopeTriangle}[6]
{

  \pgfplotsextra
  {
    \pgfkeysgetvalue{/pgfplots/xmin}{\xmin}
    \pgfkeysgetvalue{/pgfplots/xmax}{\xmax}
    \pgfkeysgetvalue{/pgfplots/ymin}{\ymin}
    \pgfkeysgetvalue{/pgfplots/ymax}{\ymax}

    \pgfmathsetmacro{\xArel}{#1}
    \pgfmathsetmacro{\yArel}{#3}
    \pgfmathsetmacro{\xBrel}{#1-#2}
    \pgfmathsetmacro{\yBrel}{\yArel}
    \pgfmathsetmacro{\xCrel}{\xArel}

    \pgfmathsetmacro{\lnxB}{\xmin*(1-(#1-#2))+\xmax*(#1-#2)} 
    \pgfmathsetmacro{\lnxA}{\xmin*(1-#1)+\xmax*#1} 
    \pgfmathsetmacro{\lnyA}{\ymin*(1-#3)+\ymax*#3} 
    \pgfmathsetmacro{\lnyC}{\lnyA+#4*(\lnxA-\lnxB)}
    \pgfmathsetmacro{\yCrel}{\lnyC-\ymin)/(\ymax-\ymin)} 

    \coordinate (A) at (rel axis cs:\xArel,\yArel);
    \coordinate (B) at (rel axis cs:\xBrel,\yBrel);
    \coordinate (C) at (rel axis cs:\xCrel,\yCrel);

    \draw[#5]   (A)-- 
    (B)-- 
    (C)-- node[pos=0.5,anchor=west] {$#6$}
    cycle;
  }
}

\usepackage[textsize=small,obeyFinal]{todonotes}

\begin{document}

\title[Dirichlet Optimal Control on Convex Domains]{Numerical Analysis for Dirichlet Optimal Control Problems on Convex Polyhedral Domains}


\author[1]{\fnm{Johannes} \sur{Pfefferer}}\email{johannes.pfefferer@unibw.de}

\author*[2]{\fnm{Boris} \sur{Vexler}}\email{vexler@tum.de}

\affil[1]{\orgdiv{Fakult\"at f\"ur Elektrische Energiesysteme und Informationstechnik}, \orgname{Universit\"at der Bundeswehr M\"unchen}, \orgaddress{\street{Werner-Heisenberg-Weg 39}, \city{Neubiberg}, \postcode{85579}, \country{Germany}}}
\affil[2]{\orgdiv{Department of Mathematics}, \orgname{Technical University of Munich, School of Computation, Information and Technology}, \orgaddress{\street{Boltzmannstr. 3}, \city{Garching b. Munich}, \postcode{85748}, \country{Germany}}}


\abstract{In this paper error analysis for finite element discretizations of  Dirichlet boundary control problems is developed. For the first time, optimal discretization error estimates are established in the case of three dimensional polyhedral and convex domains. The convergence rates solely depend on the size of largest interior edge angle. These results are comparable to those for the two dimensional case. However, the approaches from the two dimensional setting are not directly extendable such that new techniques have to be used. The theoretical results are confirmed by numerical experiments.}




\maketitle

\section{Introduction}\label{sec1}

In this paper we consider the following optimal control problem with the control entering the Dirichlet boundary condition of a linear elliptic equation: 
\begin{subequations}\label{DirichletCon:eq:problem}
	\begin{equation}\label{DirichletCon:eq:cost}
		\text{Minimize } J(q,u) = \half\ltwonorm{u-u_d}^2 + \frac{\alpha}{2}\ltwonormdo{q}^2
	\end{equation}
	over control $q$ and state $u$ fulfilling the state equation
	\begin{equation}\label{DirichletCon:eq:state}
		\begin{aligned}
			-\Lap u &= 0 &\quad&\text{in } \Omega,\\
			u &=q &\quad&\text{on } \partial \Omega
		\end{aligned}
	\end{equation}
and pointwise control constraints
\begin{equation}\label{DirichletCon:eq:constraints}
	q_a \le q(s) \le q_b \quad \text{for almost all } s \in \partial \Omega.
\end{equation}
\end{subequations}
The precise functional analytic setting is discussed below.

Such problems are often referred as \textit{Dirichlet control problems}. The inherent difficulty in the mathematical treatment of such problems, compared to distributed (control on the right-hand side) or Neumann boundary control, is the fact that the control does not directly enter the standard variational formulation of the state equation.
There are several contributions to the analysis and numerical analysis of such problems. In the majority of the publications on numerical analysis of such problems the computational domain $\Omega$ is taken as a two-dimensional polygonal domain, see, e.g., \cite{CasasRaymond:2006,MayRannacherVexler:2013}, where $\Omega$ is in addition assumed to be convex. In \cite{ApelMateosPfeffererRoesch:2015,ApelMateosPfeffererRoesch:2018} also non-convex polygonal domains are treated. In \cite{DeckelnichGuentherHinze:2009} Dirichlet control problems on smooth two- and three-dimensional domains are considered. The goal of our paper is to provide an a priori error analysis for the finite element discretization of \eqref{DirichletCon:eq:problem} in the case of three-dimensional convex polyhedral domains. A direct extension of the approaches from literature treating the two-dimensional case does not seem to be possible and thus, new techniques have to be developed. The main tools we use are estimates in weighted spaces, where the weight is a (smoothed) regularized distance to the boundary. The main contributions of the paper are the following:
 \begin{itemize}
 	\item \textit{Normal trace theorem on convex polyhedral domains.} It is well-known, that for a smooth domain $\Omega$ the normal trace $\partial_n v$ of a function $v$ in the Sobolev space $W^{2,p}(\Omega)$ possesses the regularity $W^{1-\frac{1}{p},p}(\partial \Omega)$ for every $1<p<\infty$. For a polyhedral domain this is in general not true, since the normal direction $n$ is not continuous. We check, using a general theory from \cite{MazyaMitreaShaposhnikova:2010}, that if $v \in W^{2,p}(\Omega)\cap W^{1,p}_0(\Omega)$, then $\partial_n v \in W^{1-\frac{1}{p},p}(\partial \Omega)$ for any (convex) polyhedral domain, see \cref{PDE:theorem:normaltrace_estimate} below. Such a result is available for two dimensional domains, see \cite[Lemma A.2]{CasasMateosRaymond:2009}. For three dimensional domains we could not locate it in the literature and thus provide a proof. This trace estimates is required in order to obtain the precise regularity of the (very weak) solution of the state equation \eqref{DirichletCon:eq:state}.
 	\item \textit{Weighted $H^1$ regularity of the discrete harmonic extension.} On the continuous level the very weak solution $u$ of the state equation \eqref{DirichletCon:eq:state} for a given control $q \in L^2(\partial \Omega)$ possesses the weighted regularity
 	\[
 	\ltwonorm{\rho^\half \nabla u} \le c \ltwonormdo{q},
 	\]
 	where $\rho(x)$ is the distance function to the boundary $\partial \Omega$, see the discussion and references below. We provide a corresponding result on the discrete level, i.e., for the finite element solution $u_h$ of the corresponding discretized equation we prove
 	\begin{equation}\label{eq:weightedH1_discrete_apriori}
 	 \ltwonorm{\tilde \rho^\half \nabla u_h} \le c \ltwonormdo{P_h^\partial q},
 	\end{equation}
 	where $\tilde \rho$ is a regularized distance function to the boundary, see \eqref{eq:tilde_rho} below, and $P_h^\partial$ is the $L^2$ projection on the boundary. This result is given in \cref{theorem:discrete_state_stab}. We use it in our error analysis for the optimal control problem. However, we think it is of an independent interest.
 	\item \textit{Error estimates for the Dirichlet control problem.} We provide a priori error estimates for the error between the optimal solutions $\oq$  of \eqref{DirichletCon:eq:problem} and $\oq_h$ 
 	 of corresponding discretized problems. More precisely, we will establish error estimates for two different types of discretization concepts, the concept of variational discretization and the concept of piecewise linear and continuous discretization, see \cref{theorem:q_error_est_var} and \cref{theorem:q_error_est_p1}. Under the regularity assumption $u_d \in H^1(\Omega)$ on the desired state we prove the estimate
 	\begin{equation}\label{into:main_res}
 	\ltwonormdo{\oq-\oq_h} \le ch^{\half + s}\norm{u_d}_{H^1(\Omega)}
 	\end{equation}
 	 for every 
 	 \[
 	 	s<\min\left(\frac{\pi}{\omega_\Omega}-1,\frac{1}{2}\right),
 	 \]
 	 where $\omega_\Omega<\pi$ is the largest interior edge angle (a precise definition is given below). For $\omega_\Omega<\frac23\pi$ (which is the limiting angle in the above inequality) we even show first order convergence (up to a $\log h$-factor). Please note that the convergence rates only depend on the geometry of the domain. Even more, only edge openings (and not the geometry of vertices) play a role for the convergence rates. 
 	 These results are comparable with those for the two-dimensional case from \cite{ApelMateosPfeffererRoesch:2018} (under slightly different assumptions on $u_d$). For the three-dimensional case the convergence results are absolutely new. We also note that a corresponding result for convex polygonal (i.e. two-dimensional) domains can be shown by the techniques of the present paper as well. In this regard, our new approach seems to be more flexible.
 \end{itemize}

The structure of the paper is as follows: In \cref{sec2} we collect and prove different trace estimates and regularity results for the solutions to the state equations, which are needed later on. \cref{sec3} is concerned with the discussion of the infinite dimensional optimal control problem, the derivation of corresponding optimality conditions and the elaboration of regularity results for the optimal solution. In \cref{sec4} the discretization of the state equation and related error estimates are considered in detail. As main result of this section we prove the weighted a priori estimate \cref{eq:weightedH1_discrete_apriori} for the discrete solution of the state equation. The numerical analysis for two different types of discretization concepts applied to the optimal control problem (variational discretization and piecewise linear and continuous discretization) is contained in \cref{sec5}. This paper ends with numerical examples in \cref{sec6}, which illustrate the theoretical findings of Section 5, and an acknowledgment in \cref{sec7}.

\section{Trace estimates and regularity results for the state equation}\label{sec2}
Throughout the paper the domain $\Omega \subset \R^3$ is assumed to be polyhedral and convex. The set of faces is denoted by ${\cal F}(\Omega)$, the set of all edges by ${\cal E}(\Omega)$, and the set of all vertices by ${\cal V}(\Omega)$. For every edge $e\in {\cal E}(\Omega)$ we denote by $\omega_e$ the the interior angle at~$e$. We define
\begin{equation}\label{eq:omega_Omega}
\omega_\Omega = \max_{e \in {\cal E}(\Omega)} w_e
\end{equation}
and have $\omega_\Omega < \pi$ by the convexity of $\Omega$. Moreover, we will use the critical exponent $\lambda_\Omega>1$ defined as
\begin{equation}\label{eq:lambda_Omega}
\lambda_\Omega = \frac{\pi}{\omega_\Omega}.
\end{equation}

We use the standard notation for the Lebesgue $L^p(\Omega)$ and Sobolev spaces $W^{k,p}(\Omega)$ as well as the corresponding fractional spaces $W^{s,p}(\Omega)$ equipped with Sobolev-Slobodeckij norms, see, e.g., \cite[Chapter 1.3]{Grisvard:1985} and \cite[Chapter 5]{Adams:2003} for details. We will use also weighted Sobolev spaces with a weight being a power of the distance
to the boundary $\partial \Omega$. The function $x \mapsto \dist(x,\partial \Omega)$ is Lipschitz continuous with the Lipschitz constant $L=1$, but it is in general not in $W^{2,\infty}(\Omega)$, as the boundary $\partial \Omega$ is polyhedral. For the definition of the weighted spaces this Lipschitz continuity is sufficient, but for weighted superconvergence estimates (see below) we require a smoothed weight. By \cite[Theorem 2, page 171]{Stein:1970} there exists a function $\rho \colon \bar \Omega\to [0,\infty)$ with $\rho \in C^\infty(\Omega)$ fulfilling 
\[
c_1 \rho(x) \le \dist(x,\partial \Omega) \le c_2 \rho(x) \quad \text{for all } x \in \bar \Omega,
\]
with some constants $c_1$ and $c_2$ as well as
\[
\abs{D^\alpha \rho(x)} \le c \rho(x)^{1-\abs{\alpha}} \quad \text{for all } x \in \Omega,
\]
for every multi-index $\alpha \in \N_0^3$. 

For $m \in \N_0$, $1<p<\infty$, and $s\in\R$ the space $W^{m,p}_s(\Omega)$ is the space of functions with the finite norm
\begin{equation}\label{PDE:eq:norm_weighted_w_space}
	\norm{v}_{W^{m,p}_s(\Omega)}^p = \sum_{\abs{\alpha}\le m} \into \abs{D^\alpha v(x)}^p \rho(x)^{sp}\, dx,
\end{equation}
see \cite{MazyaMitreaShaposhnikova:2010} for details. For $s=0$ this space obviously coincides  with $W^{m,p}(\Omega)$.

We will use the following Hardy type inequalities.

\begin{proposition}\label{theorem:HardyFractional}
	Let $0<s<\half$. There is a constant $c>0$ independent of $v$ such that the inequality
	\[
	\ltwonorm{\rho^{-s} v} \le c \norm{v}_{H^s(\Omega)}
	\]
	holds for all $v \in H^s(\Omega)$.
\end{proposition}
\begin{proof}
	We refer to \cite[Theorem 1.4.4.3]{Grisvard:1985}, see also \cite[(17)]{Dyda:2004} and \cite[Corollary 2.4]{ChenSong:2003}.
\end{proof}

\begin{proposition}\label{theorem:HardyType_s}
	Let $1<p<\infty$, $0<s <\frac{1}{p}$. There is a constant $c>0$ only depending on $\Omega$ such that the following inequality holds
	\[
	\lpnorm{\rho^{-s} v} \le \frac{cp}{1-sp} \left(\lpnorm{\rho^{1-s}v}+ \lpnorm{\rho^{1-s}\nabla v}\right)
	\]
	for all $v \in W^{1,p}_{1-s}(\Omega)$.
\end{proposition}

\begin{proof}
	The estimate is given in \cite[Theorem 2.4, page 293]{Necas:2012}. The dependence of the constant on $s$ and $p$ can be tracked from the proof.
\end{proof}

 For a face $F \in \mathcal{F}(\partial \Omega)$ we will also use standard notion of Sobolev spaces $W^{s,p}(F)$ for $s\in \R$ and $1\le p \le \infty$. For $0\le s\le 1$ the definition of these spaces can be directly extended to the whole boundary leading to the spaces  $W^{s,p}(\partial \Omega)$. The space of functions with face-wise regularity is defined as
\begin{equation}\label{PDE:eq:defSobolev_pw}
	W^{s,p}_{\mathrm{pw}}(\partial \Omega) = \Set{v\in L^p(\partial \Omega) | v|_F \in W^{s,p}(F) \text{ for all } F \in {\mathcal F}(\partial \Omega)}.
\end{equation}

We will use standard trace theorems, which provide continuity for the trace operator $\tau \colon W^{1,p}(\Omega) \to W^{1-\frac{1}{p},p}(\partial \Omega)$, see, e.g., \cite[Theorem 1.5.1.3]{Grisvard:1985}. The question on regularity of the normal flux $\partial_n v$ for $v \in W^{2,p}(\Omega)$ is more involved. For smooth domains the regularity $\partial_n v \in W^{1-\frac{1}{p},p}(\partial \Omega)$ is well known. On polyhedral domains such a regularity can not be expected in general, since the normal direction $n$ is discontinuous. However, we require such regularity only for functions with zero Dirichlet trace, i.e. from $W^{2,p}(\Omega)\cap W^{1,p}_0(\Omega)$. Also for $v$ from this space the desired regularity can not be expected for a general Lipschitz domain. Below, we provide a corresponding regularity result for polyhedral domains. To this end we first introduce an additional space for functions on the boundary $\partial \Omega$ in order to precisely capture the behavior of the normal trace $\partial_n v$ of a function $v$ from a Sobolev space on $\Omega$ with a zero trace. For $0<s<1$ and $1<p<\infty$ we define
\begin{equation}\label{PDE:eq:def:Nsp}
	N^{s,p}(\partial \Omega) = \Set{g \in L^p(\partial \Omega) | g\, n_i \in W^{s,p}(\partial \Omega), \, 1 \le i \le 3},
\end{equation}
where $n_i$ are the components of the normal vector $n$. The norm on this space is defined as
\[
\norm{g}_{N^{s,p}(\partial \Omega)}^p = \sum_{i=1}^N \norm{g\, n_i}^p_{W^{s,p}(\partial \Omega)}.
\]
The following proposition describes the behavior of the normal trace on $W^{2,p}_s(\Omega)\cap W^{1,p}_0(\Omega)$, see  \eqref{PDE:eq:norm_weighted_w_space} for the definition of this weighted space.
\begin{proposition}\label{PDE:theorem:normaltrace:Nsp}
	Let $1<p<\infty$ and $-\frac{1}{p}<s<1-\frac{1}{p}$. The image of the mapping $\gamma_1 \colon v \mapsto \partial_n
	v$ considered on $W^{2,p}_s(\Omega)\cap W^{1,p}_0(\Omega)$ is the space $N^{1-s-\nicefrac{1}{p},p}(\partial \Omega)$. The operator $\gamma_1 \colon W^{2,p}_s(\Omega)\cap W^{1,p}_0(\Omega) \to N^{1-s-\nicefrac{1}{p},p}(\partial \Omega)$ is continuous.
\end{proposition}
\begin{proof}
	This statement (which is true on a general Lipschitz domain) follows
	from \cite[Theorem 7.8]{MazyaMitreaShaposhnikova:2010} with $m=2$ using the fact that
	all tangential derivatives vanish. For $s=0$ and $p=2$ this result is explicitly formulated in \cite[Lemma
	6.3]{GesztesyMitrea:2011}, where also the corresponding space $N^{\nhalf}(\partial \Omega)$ is introduced, cf. also \cite{GeymonatKrasucki:2000} and \cite{DuranMuschietti:2001} for a two-dimensional domain.  
\end{proof}

In the following, we will check
that $N^{s,p}(\partial \Omega) \hookrightarrow W^{s,p}(\partial \Omega)$ on convex and polyhedral domains, which allows to prove the desired regularity for the normal trace.

\begin{lemma}\label{PDE:theorem:Nsp_into_Wsp}
	Let $0<s<1$ and $1<p<\infty$. Then there holds
	\[
	N^{s,p}(\partial \Omega) \hookrightarrow W^{s,p}(\partial \Omega).
	\]
\end{lemma}
\begin{proof}
	First we check, that $N^{s,p}(\partial \Omega) \subset W^{s,p}_{\mathrm{pw}}(\partial \Omega)$. Let $g \in N^{s,p}(\partial \Omega)$. 
	By the definition \eqref{PDE:eq:def:Nsp} of $N^{s,p}(\partial \Omega)$ we have that 
	$g n_i  \in W^{s,p}(\partial \Omega)$ for $1\le i\le 3$. Thus, for every vector $\gamma \in \R^3$ we obtain 
	\[
	(\gamma \cdot n) g =  \sum_{i=1}^3 \gamma_i (g n_i ) \in W^{s,p}(\partial \Omega).
	\]
	Let $F \in {\mathcal F}(\partial \Omega)$ be a face with the unit outer normal vector $n_F$. Choosing $\gamma = n_F$ we obtain that $\gamma \cdot n = 1$ on $F$ and therefore $g \in W^{s,p}(F)$ with
	\[
	\norm{g}_{W^{s,p}(F)} \le c \norm{g}_{N^{s,p}(\partial \Omega)}. 
	\]
	
	Since $F$ was arbitrary we have $g \in W^{s,p}_{\mathrm{pw}}(\partial \Omega)$.
	
	Let again $g \in N^{s,p}(\partial \Omega)$ be arbitrary. By the previous argument $g$ is especially in $L^p(\partial \Omega)$. In order to prove $g \in W^{s,p}(\partial \Omega)$  we have to check, that
	\[
	I = \intdo \intdo \frac{\abs{g(x)-g(y)}^p}{\abs{x-y}^{2+sp}}\,dx\,dy
	\]
	is finite. We obtain
	\[
	I = \sum_{F \in {\mathcal F}(\partial \Omega)} \sum_{G \in {\mathcal F}(\partial \Omega)} \int_{F} \int_{G} \frac{\abs{g(x)-g(y)}^p}{\abs{x-y}^{2+sp}}\,dx\,dy.
	\]
	The contributions for $F=G$ are bounded by $\norm{g}^p_{W^{s,p}_{\mathrm{pw}}(\partial \Omega)}$. For
	$F$ and $G$ with $\bar F \cap \bar G = \emptyset$, we have
	\[
	\dist(\bar F,\bar G) = \delta_{FG} >0
	\]
	and therefore
	\[
	\int_{F} \int_{G} \frac{\abs{g(x)-g(y)}^p}{\abs{x-y}^{2+sp}}\,dx\,dy \le c \norm{g}^p_{L^p(\partial \Omega)}.
	\]
	Thus, it remains to bound the contributions for $F \neq G$ with $\bar F \cap \bar G \neq \emptyset$. Let $n_F$ and
	$n_G$ be the corresponding normal vectors. Since $F$ and $G$ are neighboring faces of a convex domain, the vectors $n_F$ and $n_G$ are linearly independent. This allows us to choose a vector $\gamma = n_F + n_G \in \R^3$, $\gamma \neq 0$ with the following properties
	\[
	\gamma \cdot n_F = \gamma \cdot n_G = C_{FG} \neq 0.
	\]
	From the fact that $g \in N^{s,p}(\partial \Omega)$, we know that
	$
	g \, n_i \in W^{s,p}(\partial \Omega)
	$
	for $i=1,2,3$, where $n_i$ denotes the components of the normal vector $n$. We consider a new function $w$ given as
	\[
	w = \sum_{i=1}^N \gamma_i (g\, n_i ) = (\gamma \cdot n) g.
	\]
	There holds by construction $w \in W^{s,p}(\partial \Omega)$ and therefore especially
	\[
	\int_{F} \int_{G} \frac{\abs{w(x)-w(y)}^p}{\abs{x-y}^{2+sp}}\,dx\,dy \le c \norm{g}_{N^{s,p}(\partial \Omega)}
	\]
	is finite. Since there holds $w = (\gamma \cdot n) g$ with $\gamma \cdot n = C_{FG}\neq
	0$ on  $F \cup G$, the contribution
	\[
	\int_{F} \int_{G} \frac{\abs{g(x)-g(y)}^p}{\abs{x-y}^{2+sp}}\,dx\,dy.
	\]
	is finite and bounded by $\norm{g}_{N^{s,p}(\partial \Omega)}$ as well. This completes the proof.
\end{proof}

As a corollary we obtain the following trace theorem for the normal trace.

\begin{theorem}\label{PDE:theorem:normaltrace_estimate}
	Let $v \in W^{2,p}(\Omega)\cap
	W^{1,p}_0(\Omega)$ for some $1 < p < \infty$. Then, the normal derivative $\partial_n v$ on
	$\partial \Omega$ is well defined (in the trace sense) with 
	\[
	\partial _n v \in W^{1-\nicefrac{1}{p},p}(\partial \Omega).
	\]
	There is a constant $c>0$ independent of $v$ such that
	\[
	\norm{\partial_n v}_{W^{1-\nicefrac{1}{p},p}(\partial \Omega)} \le c \norm{v}_{W^{2,p}(\Omega)}.
	\]
\end{theorem}
\begin{proof}
	Let $v \in W^{2,p}(\Omega)\cap W^{1,p}_0(\Omega)$. From \cref{PDE:theorem:normaltrace:Nsp} for $s = 0$, we have $\partial_n v \in  N^{1-\nicefrac{1}{p},p}(\partial \Omega)$ and there holds
	\[
	\norm{\partial_n v}_{N^{1-\nicefrac{1}{p},p}(\partial \Omega)} \le c \norm{v}_{W^{2,p}(\Omega)}.
	\]
	Then by the embedding 
	$
	N^{1-\nicefrac{1}{p},p}(\partial \Omega) \hookrightarrow W^{1-\nicefrac{1}{p},p}(\partial \Omega)
	$
	from \cref{PDE:theorem:Nsp_into_Wsp} we obtain the desired result.
\end{proof}

A standard variational formulation of the state equation \eqref{DirichletCon:eq:state} requires $q \in H^{\nhalf}(\partial \Omega)$. For a functional analytic formulation of the optimal control problem, we require a formulation, which allows for $q \in L^2(\partial \Omega)$. We will use the so called very weak formulation, which is well posed even for $q \in H^{-\nhalf}(\partial \Omega)$, where $H^{-\nhalf}(\partial \Omega)$ is the dual space of $H^{\nhalf}(\partial \Omega)$. For given  $q \in H^{-\nhalf}(\partial \Omega)$ we call $u$ a very weak solution of \eqref{DirichletCon:eq:state}, if
\begin{equation}\label{PDE:eq:very_weak_inhDirichlet}
	u \in L^2(\Omega) \quad:\quad (u,-\Lap \phi) = -\langle q,\partial_n \phi \rangle_{\partial \Omega} \quad \text{for all } \phi \in H^2(\Omega) \cap H^1_0(\Omega).
\end{equation}
The duality paring between $H^{\nhalf}(\partial \Omega)$ and its dual is denoted by $\langle \cdot,\cdot \rangle_{\partial \Omega}$.
\begin{remark}\label{PDE:remark:veryweak_as_extension_of_weak}
	It is straightforward to check, that for $q \in H^{\nhalf}(\partial \Omega)$ the standard variational solution $u \in H^1(\Omega)$ of \eqref{DirichletCon:eq:state}  fulfills the very weak formulation \cref{PDE:eq:very_weak_inhDirichlet}.
\end{remark}

The existence, uniqueness, and also additional regularity $u \in H^{\nhalf}(\Omega)$ if $q \in L^2(\partial \Omega)$ of a very weak solution in the above sense is well known for smooth domains. Such results for polygonal (and even non-convex) domains can be found in \cite{ApelNicaisePfefferer:2016}. The result below holds for three dimensional convex polyhedral domains. 
\begin{theorem}\label{theorem:state}
Let $q \in H^{-\nhalf}(\partial \Omega)$. There exists a unique very weak solution of \eqref{DirichletCon:eq:state} in the sense of \eqref{PDE:eq:very_weak_inhDirichlet}. 
\begin{enumerate}
\item There is a constant $c$ depending only on $\Omega$ such that the following estimate holds
\[
\ltwonorm{u} \le c \norm{q}_{H^{-\nhalf}(\partial \Omega)}.
\]
\item If $q\in L^2(\partial \Omega)$, then $u \in H^{\nhalf}(\Omega)$ and there holds
\[
\ltwonorm{\rho^\half \nabla u} + \norm{u}_{H^{\nhalf}(\Omega)} \le c \ltwonormdo{q}.
\]
\item If $q \in H^{\nhalf}(\partial \Omega)$ then $u \in H^1(\Omega)$ and there holds
\[
\norm{u}_{H^1(\Omega)} \le c \norm{q}_{H^{\nhalf}(\partial \Omega)}.
\]
\end{enumerate}
\end{theorem}
\begin{proof}
The existence, uniqueness and the estimate for $\ltwonorm{u}$ is shown by the Riesz representation theorem. To this end we consider a linear functional $G \colon L^2(\Omega) \to \R$ defined in the following way. Let $\psi \in L^2(\Omega)$ and let $\phi$ be the weak solution of 
\[
	\begin{aligned}
	-\Lap \phi &= \psi &\quad&\text{in } \Omega,\\
	\phi &=0 &\quad&\text{on } \partial \Omega.
\end{aligned}
\]
We set $G(\psi) = -\langle q,\partial_n \phi \rangle_{\partial \Omega}$. By elliptic regularity (due  to convexity of the domain, see \cite[Theorem 3.2.1.2]{Grisvard:1985}) we have $\phi \in H^2(\Omega)\cap H^1_0(\Omega)$ and by the trace estimate from \cref{PDE:theorem:normaltrace_estimate} for $p=2$ we get $\partial_n \phi \in H^{\nhalf}(\partial \Omega)$ with
\[
\norm{\partial_n \phi}_{H^{\nhalf}(\partial \Omega)} \le c \htwonorm{\phi} \le c \ltwonorm{\psi}.
\]
There holds
\[
G(\psi) = -\langle q,\partial_n \phi \rangle_{\partial \Omega} \le \norm{q}_{H^{-\nhalf}(\partial \Omega)} \norm{\partial_n \phi}_{H^{\nhalf}(\partial \Omega)} \le c \norm{q}_{H^{-\nhalf}(\partial \Omega)} \ltwonorm{\psi}.
\]
Thus, $G$ is a continuous linear functional, and Riesz representation theorem yields the result.
The estimate for $q\in H^{\nhalf}(\partial \Omega)$ is standard and the result for $q \in L^2(\partial \Omega)$ is obtained by interpolation between the two previous estimates. 
It remains to discuss the weighted estimate. It follows by \cite[Theorem 4.1]{JerisonKenig:1995} from the regularity $u \in H^{\nhalf}(\Omega)$ and the fact that $u$ is harmonic in $\Omega$.
\end{proof}

In the sequel we also require regularity results for the equation with homogeneous Dirichlet boundary conditions to be applied for the adjoint equation. We consider the weak solution $z \in H^1_0(\Omega)$ of
\begin{equation}\label{eq:z_rhs_f}
 \begin{aligned}
 	-\Lap z &= f &\quad&\text{in } \Omega,\\
 	z &=0 &\quad&\text{on } \partial \Omega.
 \end{aligned}
\end{equation}

By convexity we know $z \in H^2(\Omega)\cap H^1_0(\Omega)$ if $f \in L^2(\Omega)$. Also a $W^{2,p}(\Omega)$ regularity result for $2 \le p < p_\Omega$ with some $p_\Omega > 2$ is well known. In addition, we will require regularity results in $C^{1,\gamma}(\bar \Omega)$ as well as in the weighted space $W^{2,2}_{-s}(\Omega)$.

\begin{proposition}\label{prop:Hoelder}
	Let $f \in L^p(\Omega)$ with some $p>3$. Then, there is some $\gamma>0$ such that the solution $z$ to \eqref{eq:z_rhs_f} possesses the regularity $z \in C^{1,\gamma}(\bar \Omega)$. There holds the estimate
	\[
	\norm{z}_{C^{1,\gamma}(\bar \Omega)} \le c \lpnorm{f}
	\]
	with a constant $c>0$ independent of $f$.
\end{proposition}
\begin{proof}
 The result follows by H\"older regularity for derivatives of the Green's function from \cite[Theorem 1]{GuzmanLeykekhmanannRossmannSchatz:2009}.
\end{proof}

To formulate the next result we define
\begin{equation}\label{eq:s_Omega}
	s_\Omega = \min\left(\lambda_\Omega-1,\half\right)\in\bigg(0,\frac12\bigg],
\end{equation}
where $\lambda_\Omega=\pi/\omega_\Omega\in(1,\infty)$ is the critical exponent and $\omega_\Omega\in(0,\pi)$ is the largest interior edge angle, see also \cref{eq:omega_Omega} and \cref{eq:lambda_Omega}.
	
\begin{theorem}\label{theorem:weighted_H2reg}
	Let $0<s_\Omega \le \frac{1}{2}$ be defined in \eqref{eq:s_Omega}. For every $0<s<s_\Omega$ and $f \in H^1(\Omega)$ the solution $z$ to \cref{eq:z_rhs_f} possesses the regularity $z \in W^{2,2}_{-s}(\Omega)$. There holds the estimate
	\[
	\norm{z}_{W^{2,2}_{-s}(\Omega)} \le \frac{c}{1-2s} \norm{f}_{H^1(\Omega)}.
	\]
	with a constant $c>0$ only depending on $\Omega$.
\end{theorem}
\begin{proof}
	There holds
	\[
	\norm{z}_{W^{2,2}_{-s}(\Omega)} \le \ltwonorm{\rho^{-s} z} + \ltwonorm{\rho^{-s} \nabla z} +\ltwonorm{\rho^{-s} \nabla^2 z}.
	\]
	Since $s<\half$ we obtain by \cref{theorem:HardyType_s}
	\[
	\ltwonorm{\rho^{-s} z} \le \frac{c}{1-2s} \left(\ltwonorm{\rho^{1-s} z} +\ltwonorm{\rho^{1-s}\nabla z}\right) \le \frac{c}{1-2s} \norm{z}_{H^1(\Omega)},
	\]
    and
	\[
	\ltwonorm{\rho^{-s} \nabla z} \le \frac{c}{1-2s} \left(\ltwonorm{\rho^{1-s}\nabla z}+\ltwonorm{\rho^{1-s}\nabla^2 z}\right) \le \frac{c}{1-2s} \htwonorm{z}.
	\]
	Thus, by $H^2$ regularity  we get
	\[
	\ltwonorm{\rho^{-s} z} + \ltwonorm{\rho^{-s} \nabla z} \le \frac{c}{1-2s} \htwonorm{z} \le \frac{c}{1-2s} \ltwonorm{f}.
	\]
	For the term involving weighted second derivatives we obtain  again by \cref{theorem:HardyType_s} since $s<\half$
	\[
	\begin{aligned}
		\ltwonorm{\rho^{-s} \nabla^2 z} &\le \frac{c}{1-2s}\left(\ltwonorm{\rho^{1-s}\nabla^2 z} +\ltwonorm{\rho^{1-s}\nabla^3 z}\right)\\
		& \le \frac{c}{1-2s}\left( \htwonorm{z} +  \ltwonorm{\rho^{1-s}\nabla^3 z}\right).
	\end{aligned}
	\]
	To estimate the term on the right-hand side involving third derivatives, we will use results from \cite{MazyaRossmann:2010}. The weighted spaces there are defined using the distance functions to the vertices and edges of the domain. Let $r$ be the distance function to the set of edges of $\Omega$. Thus, $\rho(x) \le r(x)$ for all $x \in \Omega$ and since $1-s>0$ we estimate
	\[
	\ltwonorm{\rho^{1-s}\nabla^3 z} \le \ltwonorm{r^{1-s}\nabla^3 z}.
	\]
	Next, we use estimates in weighted spaces form \cite[Lemma 4.3.1 and Corollary 4.1.10]{MazyaRossmann:2010}. This results in
	\[
	\ltwonorm{r^{1-s}\nabla^3 z} \le c \left(\ltwonorm{r^{-s}f} + \ltwonorm{r^{1-s}\nabla f} + \ltwonorm{z}\right)
	\]
	under the following conditions on the power of the weight. We choose in \cite[Lemma 4.3.1]{MazyaRossmann:2010} 
	\[
	l=3, \; p=2, \; \beta_j = 1-s, \; \delta_k = 1-s \quad \text{for all } j,k.
	\]
	We use the notation $\Lambda_v$ for the smallest positive eigenvalue of certain operator
	pencils for the vertex $v \in {\cal V}(\Omega)$, see \cite[Section 4.1.6]{MazyaRossmann:2010} for the precise definition. By the convexity of $\Omega$ we have $\Lambda_v>1$, see \cite[Section 4.3.1]{MazyaRossmann:2010}. The conditions of \cite[Lemma 4.3.1]{MazyaRossmann:2010} which we have to check are
	\[
	-1-\Lambda_v < \half +s < \Lambda_v \quad\text{and}\quad \abs{1+s} < \lambda_\Omega.
	\]
	The first condition is fulfilled by $0<s<\half$ and $\Lambda_v>1$ and the second one by $s<s_\Omega$ and the definition \eqref{eq:s_Omega} of $s_\Omega$. Please note, that \cite[Lemma 4.3.1]{MazyaRossmann:2010} provides regularity in the space $V^{3,2}_{\beta,\delta}$ in the notation of \cite{MazyaRossmann:2010}. Since all $\beta_j=\delta_k = 1-s$ the norm of this space gives a bound for the desired term $\ltwonorm{r^{1-s}\nabla^3 z}$ due to \cite[(3.12), page 90]{MazyaRossmann:2010}.
	Putting terms together  we obtain:
	\[
	\ltwonorm{\rho^{-s} \nabla^2 z} \le \frac{c}{1-2s} \left(\ltwonorm{r^{-s}f} + \ltwonorm{r^{1-s}\nabla f} + \htwonorm{z}\right).
	\]
	The $H^2$ norm of $z$ is estimated as before by the $H^2$ regularity. The term involving $\nabla f$ is obviously bounded by $\ltwonorm{\nabla f}$. To estimate $\ltwonorm{r^{-s}f}$  we use the H\"older inequality and obtain
	\[
	\begin{aligned}
		\ltwonorm{r^{-s}f}^2 &= \into r(x)^{-2s} f(x)^2 \,dx \le \norm{r^{-2s}}_{L^{\frac{3}{2}}(\Omega)} \norm{f^2}_{L^3(\Omega)}\\
		&= \left(\into r(x)^{-3s}dx\right)^{\frac{2}{3}} \norm{f}_{L^6(\Omega)}^2
		\le c \norm{f}_{H^1(\Omega)}^2,
	\end{aligned}
	\]
	where we used the embedding $H^1(\Omega) \hookrightarrow L^6(\Omega)$ as well as the fact that $r^{-3s}$ is integrable, since $r$ is the distance to the set of edges. This completes the proof.
\end{proof}

From the regularity $z \in W^{2,2}_{-s}(\Omega)$ from the previous theorem we get by \cref{PDE:theorem:normaltrace:Nsp} and \cref{PDE:theorem:Nsp_into_Wsp} that $\partial_n z \in H^{\nhalf+s}(\partial \Omega)$. Under an additional assumption on $\Omega$ (beyond convexity) we obtain a further regularity result for the normal trace.

\begin{corollary}\label{cor:normal_trace_in_H1}
Let $\lambda_\Omega > \frac{3}{2}$. Then in the setting of \cref{theorem:weighted_H2reg} there holds $\partial_n z \in H^1(F)$ for every face $F \in {\cal F}(\Omega)$ and 
\[
\norm{\partial_n z}_{H^1(F)} \le c \norm{f}_{H^1(\Omega)}
\]
with a constant $c>0$ independent of $f$.
\end{corollary}
\begin{proof}
Inspecting the proof of \cref{theorem:weighted_H2reg} in the case of $\lambda_\Omega >\frac{3}{2}$, we obtain the existence of some $s$ in the range $\half<s<1$ (in contrast to $s<\half$ in \cref{theorem:weighted_H2reg}) with
\[
\ltwonorm{r^{1-s}\nabla^3 u} \le c  \norm{f}_{H^1(\Omega)},
\]
where $r$ is the distance function to the set of edges of $\Omega$.
The corresponding conditions from  \cite[Lemma 4.3.1]{MazyaRossmann:2010} are then still fulfilled for some $s>\half$. Also for the corresponding terms of lower order there holds
\[
\ltwonorm{r^{-s-2}u} + \ltwonorm{r^{-s-1}\nabla u} + \ltwonorm{r^{-s}\nabla^2 u} \le c  \norm{f}_{H^1(\Omega)}.
\]
The above weighted estimate for $u$ and its derivatives imply that $u \in V^{3,2}_{\beta,\delta}(\Omega)$ with $\beta_j = 1-s$ and $\delta_k = 1-s$ for all $j,k$ in the notation of \cite[Section 4.1.2]{MazyaRossmann:2010}. Thus, we obtain for every partial derivative $\partial_{x_i} u$, that $\partial_{x_i} u \in  V^{2,2}_{\beta,\delta}(\Omega)$. The trace of $\partial_{x_i} u$ on every face $F \in {\mathcal F}(\Omega)$ fulfills $\partial_{x_i} u \in V^{\nicefrac{3}{2},2}_{\beta,\delta}(F)$, see \cite[Section 4.1.5]{MazyaRossmann:2010}. This directly implies $\partial_n u \in V^{\nicefrac{3}{2},2}_{\beta,\delta}(F)$. Since the weights $\beta_j$ and $\delta_k$ are equal, the above spaces correspond to the spaces with the only weight $r$ (distance to the edges) and one obtains $\partial_n u \in V^{\nicefrac{3}{2},2}_{1-s}(F)$ in the notation of \cite[Section 2.1.5]{MazyaRossmann:2010}. By \cite[Lemma 2.1.10]{MazyaRossmann:2010}, where an equivalent norm of $V^{\nicefrac{3}{2},2}_{1-s}(F)$ is presented, we get for every multi-index $\alpha$ with $\abs{\alpha}=1$
\[
\norm{r^{1-s}D^\alpha(\partial_n u)}_{H^{\nhalf}(F)} \le \norm{\partial_n u}_{V^{\nicefrac{3}{2},2}_{1-s}(F)} \le c  \norm{f}_{H^1(\Omega)}.
\]
Since on $F$ the distance $r$ describes the distance to the boundary of $F$ we can apply the fractional Hardy inequality from \cref{theorem:HardyFractional} leading to
\[
\norm{D^\alpha(\partial_n u)}_{L^2(F)} = \norm{r^{-(1-s)}r^{1-s}D^\alpha(\partial_n u)}_{L^2(F)} \le c\norm{r^{1-s}D^\alpha(\partial_n u)}_{H^{1-s}(F)},
\]
where we used $1-s<\half$. Thus, we get
\[
\begin{aligned}
	\norm{D^\alpha(\partial_n u)}_{L^2(F)} &\le c\norm{r^{1-s}D^\alpha(\partial_n u)}_{H^{1-s}(F)} \le c\norm{r^{1-s}D^\alpha(\partial_n u)}_{H^{\nhalf}(F)} \le c  \norm{f}_{H^1(\Omega)}.
\end{aligned}
\]
This results in $\partial_n u \in H^1(F)$ and the desired estimate holds.
\end{proof}

\section{Optimal control problem}\label{sec3}
In this section we discuss the optimal control problem \eqref{DirichletCon:eq:problem} and provide optimality conditions. Throughout we make the minimal assumption $u_d \in L^2(\Omega)$. We consider the formulation based on the control space $Q = L^2(\partial \Omega)$ corresponding to the $L^2(\partial \Omega)$ regularization term in the cost function \eqref{DirichletCon:eq:cost}. Other choices are possible. We refer to \cite{GunzburgerHouSvobodny:1991} for $H^1$ type regularization, to \cite{OfPhanSteinbach:2015,ChowdhuryGudiNandakumaran:2017,JohnSwierczynskiWohlmuth:2018,Winkler:2020} for $H^{\nhalf}$ (energy type) regularization, to \cite{Mateos:2021} for a sparsity promoting regularization, to \cite{Vexler:2007} for a problem with a finite dimensional control space and to \cite{KunischVexler:2007} for brief overview of possible formulations.

The set of admissible controls is defined by
\[
\Qad = \Set{q \in Q | q_a \le q(s) \le q_b \text{ for almost all } s \in \partial \Omega}.
\]
For simplicity of presentation and in order to avoid case distinctions we assume throughout $q_a,q_b\in\R$ satisfying
\begin{equation}\label{ass:qa_qb}
q_a < 0 < q_b.
\end{equation}
However, other choices are possible, which eventually lead to more regular solutions, see also the discussion in \cite{ApelMateosPfeffererRoesch:2015,ApelMateosPfeffererRoesch:2018}. The result from \cref{theorem:state} provides the existence of a (linear and continuous) solution operator $S \colon Q \to H^{\nhalf}(\Omega)$ with $S \colon q \mapsto u = u(q)$ solving \eqref{PDE:eq:very_weak_inhDirichlet}. Note, that due to $q\in L^2(\partial \Omega)$ the duality pairing in \eqref{PDE:eq:very_weak_inhDirichlet} can be replaced by the $L^2(\partial \Omega)$ inner product  $(q,\partial_n \phi)_{\partial \Omega}$.

 Using this solution operator we define the reduced cost functional
\[
j \colon Q \to R, \quad j(q) = J(q,Sq).
\]
Thus, we formulate the optimal control problem as
 \begin{equation}\label{DirichletCon:eq:reduced_problem}
	\text{minimize } j(q), \quad q \in \Qad.
\end{equation}
It is straightforward to check, that the reduced cost functional $j$ is continuous and strictly convex. Thus, the existence and uniqueness of an optimal solution $\oq \in \Qad$ follows by standard arguments. We refer to $\ou = S \oq$ as the optimal state.

The functional $j$ is two times Fr\'echet differentiable. For a control $q \in Q$ and a direction $\dq \in Q$ the directional derivatives are
given by
\[
j'(q)(\dq) = (u-u_d,\du) + \alpha(q,\dq)_{\partial\Om}
\]
and
\[
j''(q)(\dq,\dq) = \ltwonorm{\du}^2 + \alpha \ltwonormdo{\dq}^2,
\]
where $u = Sq$ and $\du = S\dq$. 

The necessary optimality condition for \eqref{DirichletCon:eq:reduced_problem} is given as a variational inequality: 
	\begin{equation}\label{eq:opt_cond}
\oq \in \Qad \quad:\quad	j'(\oq)(\dq-\oq) \ge 0 \quad \text{for all }\dq \in \Qad.
	\end{equation}
By convexity of $j$ this condition is also sufficient for the optimality.

To derive the optimality system we require an adjoint based representation of $j'(q)(\dq)$. To this end we introduce the corresponding adjoint equation.  For a given control $q\in Q$ and corresponding state $u = u(q) \in H^{\nhalf}(\Omega)$
we introduce the adjoint state $z = z(q) \in H^2(\Omega) \cap H^1_0(\Omega)$ as the weak solution of the adjoint equation
\[
\begin{aligned}
	-\Lap z &= u-u_d &&\quad\text{in } \Omega,\\
	z &=0 &&\quad\text{on } \partial \Omega
\end{aligned}
\]
with the corresponding weak formulation
\begin{equation}\label{DirichletCon:eq:adjoint_weak}
	z \in H^1_0(\Omega) \quad:\quad (\nabla \phi,\nabla z) = (u-u_d,\phi) \quad \text{for all } \phi \in H^1_0(\Omega).
\end{equation}

A straightforward calculation provides an expressions for the directional derivative of $j$ based on the solution of this adjoint equation, cf., e.g., \cite{CasasRaymond:2006}:
\[
j'(q)(\dq) = (\alpha q - \partial_n z,\dq)_{\partial \Omega},
\]
where  $q,\dq \in Q$, $u=u(q)\in H^{\nhalf}(\Omega)$ and $z=z(q)\in H^2(\Omega) \cap H^1_0(\Omega)$ solves the adjoint equation \eqref{DirichletCon:eq:adjoint_weak}.
\begin{remark}\label{DirichletCon:remark:dj_rep_domain}
	Let $\dq$ possess additional regularity $\dq \in H^{\nhalf}(\partial \Omega)$ and $B \colon H^{\nhalf}(\partial \Omega) \to H^1(\Omega)$ be an arbitrary extension operator, i.e. fulfilling $\tau B \dq = \dq$ with the trace operator $\tau \colon H^1(\Omega) \to H^{\nhalf}(\partial \Omega)$. Then the derivative $j'(q)(\dq)$ can be also expressed as
	\[
	j'(q)(\dq) = (u-u_d,B\dq) - (\nabla B\dq, \nabla z) + \alpha
	(q,\dq)_{\partial \Omega}.
	\]
	This expression follows from
	\[
	\begin{aligned}
		(u-u_d,\du) &= (u-u_d,\du-B\dq) + (u-u_d,B\dq)\\ 
		&= (\nabla z,\nabla(\du-B\dq)) + (u-u_d,B\dq)\\
		&= - (\nabla z,\nabla B\dq) + (u-u_d,B\dq),
	\end{aligned}
	\]
	where we used that for $\dq \in H^{\nhalf}(\partial \Omega)$ the solution $\du = S \dq$ fulfills the weak formulation
	\[
	\du \in B \dq + H^1_0(\Omega) \quad:\quad (\nabla \du,\nabla \phi) = 0 \quad \text{for all } \phi \in H^1_0(\Omega).
	\]
	On the discrete level we will use a corresponding expression for the derivative of the discrete reduced cost functional.
\end{remark}

Due to the structure of the admissible set $\Qad$ the optimality condition \cref{eq:opt_cond} can be equivalently rewritten as
\begin{equation}\label{eq:opt_cond_proj}
\oq = P_{[q_a,q_b]}\left(\frac{1}{\alpha} \partial_n \oz\right),
\end{equation}
where $\oz = z(\oq)$ is the optimal adjoint state and $P_{[q_a,q_b]}$ is the pointwise projection on the interval $[q_a,q_b]$. Since $\oz \in H^2(\Omega)\cap H^1_0(\Omega)$, $\partial_n \oz \in H^{\nhalf}(\partial \Omega)$ by \cref{PDE:theorem:normaltrace_estimate}, and since this projection is well defined as self-mapping of $H^{\nhalf}(\partial \Omega)$, cf. \cite[Lemma 3.3]{KunischVexler:2007}, we obtain additional regularity for the optimal control $\oq \in H^{\nhalf}(\partial \Omega)$. This allows to use the classical weak formulation of the state equation in the following optimality system.

\begin{theorem}\label{DirichletCon:theorem:opt_sys}
	A control $\oq\in \Qad$ is the solution of the optimal control problem
	\eqref{DirichletCon:eq:reduced_problem} if and only if the triple $(\oq,\ou,\oz)$
	with $\oq \in H^{\nhalf}(\partial \Omega)\cap \Qad$, $\ou \in B \oq + H^1_0(\Omega)$ with an arbitrary extension operator $B\colon H^{\nhalf}(\partial \Omega) \to H^1(\Omega)$ and $\oz \in H^2(\Omega)\cap H^1_0(\Omega)$ fulfills the following system:
	\begin{subequations}\label{DirichletCon:eq:opt_sys}
		\begin{itemize}
			\item State equation:
			\begin{equation}\label{DirichletCon:eq:opt_sys_state}
				(\nabla \ou,\nabla\phi)=0\quad\text{for all }\phi\in H^1_0(\Omega)
			\end{equation}
			\item Adjoint equation:
			\begin{equation}\label{DirichletCon:eq:opt_sys_adjoint}
				(\nabla \oz,\nabla \phi) =(\ou-u_d,\phi) \quad\text{for all }\phi\in H^1_0(\Omega)\\
			\end{equation}
			\item Optimality condition:
			\begin{equation}\label{DirichletCon:eq:opt_sys_gradient}
				(\alpha \oq - \partial_n \oz,\dq - \oq)_{\partial \Omega} \ge 0 \quad \text{for all } \dq \in \Qad.		
			\end{equation}
		\end{itemize}
	\end{subequations}
Moreover, there exists a constant $c>0$ only depending on $\alpha$ and $\Omega$ such that
\[
\norm{\oq}_{H^{\nhalf}(\partial \Omega)} + \norm{\ou}_{H^1(\Omega)} +  \norm{\oz}_{H^2(\Omega)} \le c \ltwonorm{u_d}.
\]
\end{theorem}

\begin{remark}
The regularity from th previous theorem can be extended to $\oq \in W^{1-\frac{1}{p},p}(\partial \Omega)$ provided $u_d \in L^p(\Omega)$ with $2 \le p < p_\Omega$ and $p_\Omega >2$ being the critical value for $W^{2,p}$ regularity. The analysis in \cite{CasasRaymond:2006} and \cite{MayRannacherVexler:2013} is based on this regularity setting.
\end{remark}

For the two dimensional case the optimal control $\oq$ is known to vanish at vertices of $\partial \Omega$. We extend this result to our three-dimensional setting. In particular, we prove that the optimal control is equal to zero on all vertices and edges of $\partial \Omega$. This is especially required when proving the $H^1(\partial\Omega)$ regularity in \cref{DirichletCon:theorem:W22-s_oz_with_s_to_12}.

\begin{theorem}\label{DirichletCon:theorem:q=0_on_edges}
	Let $u_d \in L^p(\Omega)$ with some $p>3$. Let $\oq$ be the optimal control of \eqref{DirichletCon:eq:reduced_problem}, $\ou$ be the corresponding state and $\oz$ the corresponding adjoint state. Then, we have $\oz \in C^{1,\gamma}(\bar \Omega)$ with some $\gamma >0$, and $\partial_n\oz,\oq \in C(\partial \Omega)$. Moreover, for every two faces $F,G \in {\mathcal F}(\partial
	\Omega)$, $F \neq G$ with non-empty intersection, i.e.,
	$
	\bar F \cap \bar G \neq \emptyset,
	$
	there holds
	\[
	\partial_n\oz(s)=\oq(s) = 0 \quad \text{for all } s \in \bar F \cap \bar G.
	\]
\end{theorem}
\begin{proof}
	From \cref{DirichletCon:theorem:opt_sys} we know that $\ou \in H^1(\Omega)$ and therefore we have for the right-hand side of the adjoint equation \eqref{DirichletCon:eq:opt_sys_adjoint} $\ou-u_d \in L^r(\Omega)$ with some $r>3$. Thus, by \cref{prop:Hoelder}, we get $\oz \in C^{1,\gamma}(\bar \Omega)$ with some $\gamma>0$ and especially $\nabla \oz$ is continuous on $\partial \Omega$. Since $\oz=0$ on $F$ and on $G$, we have that all tangential derivatives of $\oz$ vanish on $F$ and on $G$ and by continuity on $\bar F \cap \bar G$. Since the normal vectors $n_F$ and $n_G$ are linearly independent, this implies that $\nabla \oz$ vanishes on $\bar F \cap \bar G$. This results in $\partial_n \oz = 0$ on $\bar F \cap \bar G$ and the continuity of $\partial_n \oz$ on $\partial \Omega$. The same result for the optimal control $\oq$ follows from the optimality condition \eqref{eq:opt_cond_proj} taking into account the assumption \eqref{ass:qa_qb}, i.e. $q_a < 0 <q_b$. This completes the proof.
\end{proof} 

Based on \cref{theorem:weighted_H2reg} we provide additional regularity for the optimal control, state, and adjoint state assuming $u_d \in H^1(\Omega)$. 

\begin{theorem}\label{theorem:W22-s_oz}
	Let $u_d \in H^1(\Omega)$ and $0<s<s_\Omega$, where $s_\Omega$ is defined in  \eqref{eq:s_Omega}. Let $(\oq,\ou,\oz)$ be the solution of the optimality system
	\cref{DirichletCon:eq:opt_sys}. Then, there holds $\oz \in W^{2,2}_{-s}(\Omega)$, $\oq \in H^{\nhalf+s}(\partial \Omega)$, and $\ou \in H^{1+s}(\Omega)$. There is a constant $c>0$ independent of $u_d$ such that
	\[
	\norm{\oz}_{W^{2,2}_{-s}(\Omega)} + 
	\norm{\oq}_{H^{\nhalf+s}(\partial \Omega)} + \norm{\partial_n \oz}_{H^{\nhalf+s}(\partial \Omega)} + \norm{\ou}_{H^{1+s}(\Omega)}\le c \norm{u_d}_{H^1(\Omega)}.
	\]
\end{theorem}
\begin{proof}
	From \cref{DirichletCon:theorem:opt_sys}, we obtain $\ou \in H^1(\Omega)$
	and the corresponding estimate. Thus, it holds $\ou-u_d \in H^1(\Omega)$ with
	\[
	\norm{\ou-u_d}_{H^1(\Omega)} \le c \norm{u_d}_{H^1(\Omega)}.
	\]
	This allows for the application of the weighted regularity result from \cref{theorem:weighted_H2reg} leading to $\oz \in W^{2,2}_{-s}(\Omega) \cap H^1_0(\Omega)$ and the estimate
	\[
	\norm{\oz}_{W^{2,2}_{-s}(\Omega)} \le c \norm{\ou-u_d}_{H^1(\Omega)} \le c \norm{u_d}_{H^1(\Omega)}.
	\]
	Application of the trace estimate from \cref{PDE:theorem:normaltrace:Nsp} leads to $\partial_n \oz \in N^{\nhalf+s,2}(\partial \Omega)$ and the embedding $N^{\nhalf+s,2}(\partial \Omega) \hookrightarrow H^{\nhalf+s}(\partial \Omega)$ from \cref{PDE:theorem:Nsp_into_Wsp} results then in $\partial_n \oz \in H^{\nhalf+s}(\partial \Omega)$ with
	\[
	\norm{\oq}_{H^{\nhalf+s}(\partial \Omega)} \le c \norm{\partial_n \oz}_{H^{\nhalf+s}(\partial \Omega)} \le c \norm{\oz}_{W^{2,2}_{-s}(\Omega)} \le c \norm{u_d}_{H^1(\Omega)},
	\]
	where we used, that the projection $P_{[q_a,a_b]}$ is well defined as a self-mapping of $H^{\nhalf+s}(\partial \Omega)$, cf. \cite[Lemma 3.3]{KunischVexler:2007}. The regularity of $\ou$ is now obtained by an interpolation argument. For Dirichlet data from $H^{\nhalf}(\partial \Omega)$ the harmonic extension is in $H^1(\Omega)$ and for Dirichlet data from $H^1(\partial\Omega)$ it lies in $H^{\nicefrac{3}{2}}(\Omega)$, see \cite[Theorem 5.15 (b), $p=2$]{JerisonKenig:1995}. Thus, by interpolation we get
	\[
	\norm{\ou}_{H^{1+s}(\Omega)}\le c \norm{\oq}_{H^{\nhalf+s}(\partial \Omega)} \le c \norm{u_d}_{H^1(\Omega)}.
	\]
	This completes the proof.
\end{proof}

Under an additional assumption on $\Omega$ (beyond convexity) we obtain a further regularity result.

\begin{theorem}\label{DirichletCon:theorem:W22-s_oz_with_s_to_12}
	Let $u_d \in H^1(\Omega)$ and let $\Omega$ be such that $\lambda_\Omega > \frac{3}{2}$, where $\lambda_\Omega$ is the critical exponent defined in \eqref{eq:lambda_Omega}. Further, let $(\oq,\ou,\oz)$ be the solution of the optimality system
	\cref{DirichletCon:eq:opt_sys}. Then, for $0<s<\frac12$ there holds
	\[
	\norm{\oz}_{W^{2,2}_{-s}(\Omega)} \le \frac{c}{1-2s}\norm{u_d}_{H^1(\Omega)}
	\]
	with a constant $c$ independent of $s$ and $u_d$. Moreover, we have $\oq,\partial_n \oz \in H^1(\partial \Omega)$ and the estimate
	\[
	\norm{\oq}_{H^1(\partial \Omega)} + \norm{\partial_n \oz}_{H^1(\partial \Omega)} \le c \norm{u_d}_{H^1(\Omega)}
	\]
	holds with a constant $c>0$ independent of $u_d$.
\end{theorem}
\begin{proof}
	The proof of the first result follows the lines of the proof of \cref{theorem:W22-s_oz}. 
	Since $\lambda_\Omega > \frac{3}{2}$ we have $s_\Omega = \half$ and the estimate from \cref{theorem:weighted_H2reg} holds for all $s<\half$ leading to
	\[
	\norm{\oz}_{W^{2,2}_{-s}(\Omega)} \le \frac{c}{1-2s} \norm{u_d}_{H^1(\Omega)}.
	\]
	For the second result, we obtain by \cref{cor:normal_trace_in_H1} $\partial_n \oz \in H^1(F)$ for every $F \in {\mathcal F}(\Omega)$ and the corresponding estimate. Since $u_d \in H^1(\Omega) \hookrightarrow L^6(\Omega)$ by Sobolev embedding, we can apply \cref{DirichletCon:theorem:q=0_on_edges} leading to $\partial_n \oz \in C(\partial \Omega)$. This results in the global regularity $\partial_n \oz \in H^1(\partial \Omega)$ and the corresponding estimate holds. The result for $\oq$ follows from the optimality condition \cref{DirichletCon:eq:opt_sys_gradient}. This completes the proof.
\end{proof} 

\section{Discretization estimates for the state equation}\label{sec4}
For the discretization we consider the space of linear finite elements $\widehat V_h \subset H^1(\Omega)$ defined on a mesh $\T_h$ from a family of shape regular quasi-uniform meshes, see, e.g., \cite{BrennerScott:2008}. The mesh $\T_h = \{K\}$ consists of cells $K$, which are open tetrahedrons.
We also use the subspace $V_h = \widehat V_h \cap H^1_0(\Omega)$ with homogeneous  Dirichlet boundary conditions and the space $V_h^\partial$ of traces of functions from $\widehat V_h$, i.e.
\[
V_h^\partial = \Set{\tau v_h | v_h \in \widehat V_h},
\]
where $\tau \colon H^1(\Omega) \to H^{\nhalf}(\partial \Omega)$ is the trace operator. The state equation \eqref{PDE:eq:very_weak_inhDirichlet} is discretized as follows: For given $q \in L^2(\partial \Omega)$ the discrete state $u_h = u_h(q) \in \widehat V_h$ fulfills
\begin{equation}\label{FEM:eq:poisson_inhomogenDirichlet}
u_h \in E_h P_h^\partial q + V_h \quad:\quad (\nabla u_h,\nabla \phi_h) = 0 \quad \text{for all } v_h \in V_h,
\end{equation}
where $P_h^\partial \colon L^2(\partial \Omega) \to V_h^\partial$ is the $L^2(\partial \Omega)$ projection and $E_h \colon V_h^\partial \to \widehat V_h$ is the extension-by-zero operator. We call the operator $S_h \colon q \mapsto u_h(q)$ the discrete solution operator. For $q \in V_h^\partial$ it is often called the discrete harmonic extension.

\begin{remark}
 The above formulation is independent of the choice of the extension operator and thus, other choices are possible, which lead to equivalent formulations. The choice of the $L^2(\partial \Omega)$ projection is due to the consideration of potentially irregular boundary data and allows for optimal convergence of the finite element method under minimal regular assumptions on the data, cf. the discussions in \cite{Berggren:2004,BartelsCarstensenDolzmann:2004,ApelNicaisePfefferer:2017}.
\end{remark}
By straightforward arguments one obtains the unique solvability of \eqref{FEM:eq:poisson_inhomogenDirichlet}. Moreover, finite element error estimates can be established.
\begin{theorem}\label{theorem:state_error_est}
	Let $q \in H^s(\partial \Omega)$ with some $0 \le s \le \half$. Let $u$ be the solution of \eqref{PDE:eq:very_weak_inhDirichlet} and $u_h$ of \eqref{FEM:eq:poisson_inhomogenDirichlet}. There holds
	\[
	\ltwonorm{u-u_h} \le c h^{s+\half} \norm{q}_{H^s(\partial \Omega)}
	\]
	with a constant $c>0$ independent of $q$ and $h$.
\end{theorem}
\begin{proof}
	The proof follows standard arguments. For the convenience of the reader, we repeat the essential steps.
	Let $\tilde u\in H^1(\Omega)$ be the harmonic extension of $P_h^\partial q$, i.e., the weak solution to
	\[
		-\Delta \tilde u = 0\quad \text{in }\Omega,\quad \tilde u=P_h^\partial q\quad \text{on }\partial\Omega.
	\]
	We split the $L^2(\Omega)$-error into two,
	\begin{equation}\label{eq:state_1}
		\ltwonorm{u-u_h}\le \ltwonorm{u-\tilde u}+\ltwonorm{\tilde u-u_h}.
	\end{equation}
	We notice that $\tilde u$ is also a very weak solution, such that $u-\tilde u$ satisfies \cref{PDE:eq:very_weak_inhDirichlet} with Dirichlet data $q-P_h^\partial q$. As a consequence, we get from \cref{theorem:state}
	\[
		\ltwonorm{u-\tilde u}\le c \norm{q-P_h^\partial q}_{H^{-\nhalf}(\partial \Omega)}\le c h^{s+\half} \norm{q}_{H^s(\partial \Omega)},
	\]
	where we applied a standard estimate for the $L^2(\partial\Omega)$ projection in the last step. It remains to estimate the second term in \cref{eq:state_1}. As $\tilde u$ is also a weak solution with Dirichlet boundary data $P_h^\partial q$, we can apply a duality argument, i.e., we define $w\in H^1_0(\Omega)$ as the weak solution to
	\[
	-\Delta w = \tilde u - u_h\quad \text{in }\Omega,\quad w=0\quad \text{on }\partial\Omega
	\]
	and $w_h$ its $H^1_0(\Omega)$-Ritz-projection.
	Moreover, let $\tilde i_h:H^1(\Omega)\to \widehat V_h$ be the Scott-Zhang interpolation operator \cite{ScottZhang:1990}. This operator preserves boundary conditions from $V_h^\partial$ and admits the estimate
	\begin{equation}\label{eq:scott_zhang_stab}
		\norm{\nabla (v - \tilde i_h v)}_{L^2(\Omega)}\le c h^s \norm{v}_{H^{1+s}(\Omega)}
	\end{equation}
	for all $v\in H^{1+s}(\Omega)$ and $s\in[0,1]$. Then, we obtain
	\[
		\ltwonorm{\tilde u-u_h}^2=(\nabla (\tilde u-\tilde i_h \tilde u),\nabla(w-w_h))\le \norm{\nabla (\tilde u-\tilde i_h \tilde u)}_{L^2(\Omega)}\norm{\nabla(w-w_h)}_{L^2(\Omega)}.
	\]
	A standard finite element error estimate employing elliptic regularity in convex domains leads to
	\[
		\norm{\nabla(w-w_h)}_{L^2(\Omega)}\le c h \ltwonorm{\tilde u-u_h}.
	\]
	The stability estimate \cref{eq:scott_zhang_stab} for the Scott-Zhang interpolation operator and \cref{theorem:state} imply
	\[
		\norm{\nabla (\tilde u-\tilde i_h \tilde u)}_{L^2(\Omega)}\le c \norm{\tilde u}_{H^1(\Omega)}\le c \norm{P_h^\partial q}_{H^{\nhalf}(\partial \Omega)}. 
	\]
	As a consequence, we get
	\[
		\ltwonorm{\tilde u-u_h}\le c h \norm{P_h^\partial q}_{H^{\nhalf}(\partial \Omega)}
	\]
	and by means of an inverse estimate
	\[
		\ltwonorm{\tilde u-u_h}\le c h^{\half} \norm{P_h^\partial q}_{L^2(\partial \Omega)}.
	\]
	An interpolation argument gives
	\[
		\ltwonorm{\tilde u-u_h}\le c h^{s+\half} \norm{P_h^\partial q}_{H^{s}(\partial \Omega)}.
	\]
	The $L^2(\partial\Omega)$ projection is known to be stable in $L^2(\partial\Omega)$. Moreover, it is stable in $H^1(\partial \Omega)$ by the quasi-uniformity of the mesh, and thus also on $H^s(\partial \Omega)$ by another interpolation argument. Combining the previous results completes the proof.
\end{proof}
\begin{remark}
	The above estimate can be extended to the range $0 \le s \le \frac32$ (eventually) with a different norm on the right hand side (see e.g. \cite{BartelsCarstensenDolzmann:2004} for the case $s=\frac32$), but this is not required in our analysis.
\end{remark}

By means of \cref{theorem:state} and \cref{theorem:state_error_est}, the discrete solution operator is (uniformly in $h$) bounded as $S_h  \colon L^2(\partial \Omega) \to L^2(\Omega)$.

We require a discrete version of the weighted stability result from \cref{theorem:state}. To this end we define a regularized weight $\tilde \rho \colon \bar \Omega \to \R_+$ by
 \begin{equation}\label{eq:tilde_rho}
 	\quad \tilde \rho(x) = \sqrt{\rho(x)^2+\kappa^2 h^2},
 \end{equation} 
where $\kappa \ge 1$ is chosen later large enough independent of $h$. 
 
 \begin{theorem}\label{theorem:discrete_state_stab}
 	Let $q \in L^2(\partial\Omega)$ and $u_h = u_h(q)$ be the solution of \cref{FEM:eq:poisson_inhomogenDirichlet}. There is a constant $c>0$ independent on $h$ and $q$ such that the following estimate holds
 	\[
 	\ltwonorm{\tilde \rho^\half \nabla u_h}  \le c \ltwonormdo{P_h^\partial q}.
 	\]
 \end{theorem}
To provide a proof of this result we require some weighted interpolation and super-convergence estimates. First, we obtain by direct calculation  that 
\begin{equation}\label{eq:prop_tilde_rho}
\abs{\nabla \tilde \rho}  \le c \quad \text{and} \quad \abs{\nabla^2 \tilde \rho} \le c \tilde \rho^{-1}.
\end{equation}
A direct consequence of the mean value theorem is due to the boundedness of the gradient of $\tilde \rho$ the existence of a constant $c$ independent of $h$ such that the following estimate holds
\begin{equation}\label{eq:tilde_rho_local}
\max_{x \in \bar K} \tilde \rho(x) \le c \min_{x \in \bar K} \tilde \rho (x) \quad \text{for all } K \in \T_h.
\end{equation}

This directly leads to weighted interpolation error estimates for the Lagrange interpolation $i_h \colon C(\bar \Omega) \to \widehat V_h$, which are summarized in the following lemma.
\begin{lemma}\label{FEM:lemma:est_interpolation_general_weight}
	Let $\alpha \in \R$ be arbitrary. Then there is a constant $c>0$ independent of $\alpha$, $h$ and $v$ such that
	\[
	\ltwonorm{\tilde \rho^\alpha(v-i_h v)} + h \ltwonorm{\tilde \rho^\alpha\nabla (v-i_h v)} \le c h^2 \ltwonorm{\tilde \rho^\alpha \nabla^2 v}
	\]
	holds for all $v \in H^2(\Omega)$.
\end{lemma}

The next lemma provides a weighted inverse estimate. 
\begin{lemma}\label{FEM:lemma:weighted_inverse}
	Let $\alpha \in \R$ be arbitrary. There is a constant $c>0$ independent of $\alpha$ and $h$ such that
	\[
	\norm{\tilde \rho^\alpha \nabla v_h}_{L^2(K)} \le ch^{-1} \norm{\tilde \rho^\alpha v_h}_{L^2(K)}
	\]
	is valid for all $K\in\T_h$ and $v_h \in \widehat V_h$.
\end{lemma}
\begin{proof}
The proof follows by standard inverse estimates and \cref{eq:tilde_rho_local}.
\end{proof}

The next result is a superconvergence type estimate. A similar one for a different weight (regularized distance to a point) can be found, e.g., in \cite[Lemma 3]{LeykekhmanD_VexlerB_2016c}. 
\begin{lemma}\label{FEM:lemma:superconvergence_ih}
	Let $\alpha,\beta \in \R$ be arbitrary. There is a constant $c>0$ independent of $\alpha$, $\beta$, and $h$ such that	
	\[
	\ltwonorm{\tilde \rho^\alpha(\operatorname{id}-i_h)(\tilde \rho^\beta v_h)} + h\ltwonorm{\tilde \rho^\alpha\nabla  (\operatorname{id}-i_h)(\tilde \rho^\beta v_h)} \le c h \ltwonorm{\tilde \rho^{\alpha+\beta-1}v_h}
	\]
	holds for all $v_h \in \widehat V_h$.
\end{lemma}
\begin{proof}
	For any $K \in \T_h$, we obtain by a local interpolation estimate and \eqref{eq:prop_tilde_rho}
	\[
	\ltwonormk{\tilde \rho^\alpha(\operatorname{id}-i_h)(\tilde \rho^\beta v_h)} \le c h^2 \ltwonormk{\tilde \rho^\alpha \nabla^2(\tilde \rho^\beta v_h)}.
	\]
	Using the fact that $v_h \in \widehat V_h$ and thus affine linear on $K$ we get
	\[
	\abs{\nabla^2(\tilde \rho^\beta v_h)} \le \abs{\beta(\beta-1)} \tilde \rho^{\beta-2} \abs{\nabla \tilde \rho}^2 \abs{v_h} + \abs{\beta} \tilde \rho^{\beta-1} \abs{\nabla^2 \tilde \rho} v_h + 2 \abs{\beta} \tilde \rho^{\beta-1} \abs{\nabla \tilde \rho} \abs{\nabla v_h}.
	\]
	By the properties of $\tilde \rho$ from \cref{eq:prop_tilde_rho} we obtain
	\[
	\abs{\nabla^2(\tilde \rho^\beta v_h)} \le c \tilde \rho^{\beta-2}\abs{v_h} + c \tilde \rho^{\beta-1} \abs{\nabla v_h}.
	\]
	Thus, we have
	\[
	\ltwonormk{\tilde \rho^\alpha(\operatorname{id}-i_h)(\tilde \rho^\beta v_h)} \le c h^2 \ltwonormk{\tilde \rho^{\alpha +\beta-2}v_h} + ch^2 \ltwonormk{\tilde \rho^{\alpha+\beta-1} \nabla v_h}.
	\]
	For the first term we use $\tilde \rho \ge \kappa h$ from the definition \eqref{eq:tilde_rho} of $\tilde \rho$ and obtain
	\[
	h^2 \ltwonormk{\tilde \rho^{\alpha +\beta-2}v_h} \le ch \ltwonormk{\tilde \rho^{\alpha +\beta-1}v_h}.
	\]
	For the second term we use the weighted inverse estimate from \cref{FEM:lemma:weighted_inverse}. Thus,
	\[
	h^2 \ltwonormk{\tilde \rho^{\alpha+\beta-1} \nabla v_h} \le c h \ltwonormk{\tilde \rho^{\alpha+\beta-1} v_h}.
	\]
	Putting terms together we obtain the desired estimate. The proof for the weighted $H^1$ seminorm follows the same steps.
\end{proof}

Using the above weighted estimates we provide a proof of \cref{theorem:discrete_state_stab}.
\begin{proof}
	We define $\tilde u$ as the harmonic extension of $P_h^\partial q$, i.e., the solution to
\[
	-\Lap \tilde u = 0 \quad\text{in } \Omega,\quad 
	\tilde u = P_h^\partial q\quad\text{on } \partial \Omega.
\]
Since $P_h^\partial q \in V_h^\partial$ and $V_h^\partial \subset H^{\nhalf}(\partial \Omega)$ we have $\tilde u \in H^1(\Omega)$. By \cref{theorem:state} we get
\[
\ltwonorm{\rho^\half \nabla \tilde u} \le c \ltwonormdo{P_h^\partial q}.
\]
The proof now follows three major steps:
\begin{enumerate}
	\item In the first step, we show that we can replace $\rho$ by $\tilde \rho$ in the
	previous estimate. To this end, we consider the subdomain 
	\[
	\Omega_h = \Set{x \in \Omega | \rho(x) \le h}.
	\]
	On this subdomain we get $\tilde \rho(x) \le ch$ and therefore
	\[
	\begin{aligned}
		\norm{\tilde \rho^\half \nabla \tilde u}_{L^2(\Omega_h)}^2 &\le c h \norm{\nabla \tilde u}_{L^2(\Omega_h)}^2
		\le c h \ltwonorm{\nabla \tilde u}^2 \le ch \norm{P_h^\partial q}^2_{H^{\nhalf}(\partial \Omega)} \le c \ltwonormdo{P_h^\partial q}^2,
	\end{aligned}
	\]
	where we have used the $H^1$ estimate for $\tilde u$ from
	\cref{theorem:state} and an inverse estimate on the boundary. 
	On the compliment $\Omega \setminus \Omega_h$ we have $h \le \rho(x)$ and therefore $\tilde \rho(x) \le c \rho(x)$. Thus, \cref{theorem:state} implies
	\[
	\begin{aligned}
		\norm{\tilde \rho^\half \nabla \tilde u}_{L^2(\Omega \setminus \Omega_h)}^2 \le c \norm{\rho^\half \nabla \tilde u}_{L^2(\Omega \setminus \Omega_h)}^2
		& \le  c \ltwonorm{\rho^\half \nabla \tilde u}^2 \le c \ltwonormdo{P_h^\partial q}^2.
	\end{aligned}
	\] 
	Putting these two estimates together, we obtain
	\begin{equation}\label{FEM:eq_in_proof_harmonic_weighted:1}
		\ltwonorm{\tilde \rho^\half \nabla \tilde u} \le c \ltwonormdo{P_h^\partial q}.
	\end{equation}
	\item Since $\tilde u \in H^1(\Omega)$, we are allowed to apply
	the Scott-Zhang interpolation operator $\tilde i_h \colon H^1(\Omega) \to \widehat V_h$, see
	\cite{ScottZhang:1990}.  We will next show that the desired estimate holds for $\tilde i_h \tilde u$. Using the local behavior of the Scott-Zhang
	interpolation operator and \eqref{eq:tilde_rho_local}, we obtain for every cell $K\in \T_h$
	\[
	\norm{\tilde \rho^\half \nabla \tilde i_h \tilde u}_{L^2(K)} \le c \norm{\tilde
		\rho^\half\tilde u}_{L^2(\omega_K)} + c\norm{\tilde \rho^\half \nabla \tilde
		u}_{L^2(\omega_K)},
	\]
	where $\omega_K$ is a patch containing $K$. Summing up over all cells in the mesh and using shape regularity, we obtain
	\[
	\ltwonorm{\tilde \rho^\half \nabla \tilde i_h \tilde u} \le c \ltwonorm{\tilde \rho^\half\tilde u} + c\ltwonorm{\tilde \rho^\half \nabla \tilde u}.
	\]
	By boundedness of $\tilde \rho$ and \cref{theorem:state}, we get
	\[
	\ltwonorm{\tilde \rho^\half\tilde u} \le c \ltwonorm{\tilde u} \le c \norm{\tilde
		u}_{H^{\nhalf}(\Omega)} \le c \ltwonormdo{P_h^\partial q}.
	\]
	Using this and \eqref{FEM:eq_in_proof_harmonic_weighted:1}, we obtain
	\begin{equation}\label{FEM:eq_in_proof_harmonic_weighted:2}
		\ltwonorm{\tilde \rho^\half \nabla \tilde i_h \tilde u} \le c \ltwonormdo{P_h^\partial q}.
	\end{equation}
	\item For the last step of the proof, we set
	\[
	w_h = \tilde i_h \tilde u - u_h.
	\]
	It remains to prove an analogous estimate for $\ltwonorm{\tilde \rho^\half \nabla w_h}$. By the properties of the Scott-Zhang interpolation we have $w_h \in V_h$ (especially $w_h = 0$ on $\partial \Omega$) and
	\[
	(\nabla w_h,\nabla \phi_h) = (\nabla \tilde i_h \tilde u,\nabla \phi_h) \quad \text{for all } \phi_h \in V_h
	\]
	by equation \eqref{FEM:eq:poisson_inhomogenDirichlet} for $u_h$.  There holds
	\[
	\ltwonorm{\tilde \rho^\half \nabla w_h}^2 = (\tilde \rho \nabla w_h,\nabla w_h) = (\nabla(\tilde \rho w_h),\nabla w_h) - (w_h \nabla \tilde \rho,\nabla w_h) = I_1 + I_2.
	\]
	Since $\tilde \rho w_h$ vanishes on the boundary $\partial \Omega$ (as $w_h$), we have $i_h (\tilde \rho w_h) \in V_h$. Thus,
	\[
	\begin{split}
		I_1 &= (\nabla(\operatorname{id}-i_h)(\tilde \rho w_h),\nabla w_h) + (\nabla i_h (\tilde \rho w_h),\nabla w_h)\\
		&= (\nabla(\operatorname{id}-i_h)(\tilde \rho w_h),\nabla w_h) +  (\nabla i_h (\tilde \rho w_h),\nabla \tilde i_h \tilde u)\\
		& = I_{11} + I_{12}.
	\end{split}
	\]
	For $I_{11}$ we distribute powers of $\tilde \rho$ and obtain by the Cauchy-Schwarz inequality
	\[
	\begin{aligned}
		I_{11} &\le \ltwonorm{\tilde \rho^{-\half}\nabla(\operatorname{id}-i_h)(\tilde \rho w_h)} \ltwonorm{\tilde \rho^\half \nabla w_h}\\
		&\le \frac{1}{4} \ltwonorm{\tilde \rho^\half \nabla w_h}^2 + c \ltwonorm{\tilde \rho^{-\half}\nabla(\operatorname{id}-i_h)(\tilde \rho w_h)}^2\\
		&\le \frac{1}{4} \ltwonorm{\tilde \rho^\half \nabla w_h}^2 + c \ltwonorm{\tilde \rho^{-\half} w_h}^2,
	\end{aligned}
	\]
	where in the last step we used the superconvergence estimate from
	\cref{FEM:lemma:superconvergence_ih}. For $I_{12}$, we obtain using $\abs{\nabla \tilde \rho} \le
	c$ from \eqref{eq:prop_tilde_rho}
	\[
	\begin{aligned}
		I_{12} &= (\nabla (\tilde \rho w_h),\nabla \tilde i_h \tilde u) + (\nabla (i_h-\operatorname{id}) (\tilde \rho w_h),\nabla \tilde i_h \tilde u)\\
		& = (\tilde \rho \nabla w_h,\nabla \tilde i_h \tilde u) + (w_h \nabla \tilde \rho,\nabla \tilde i_h \tilde u) + (\nabla (i_h-\operatorname{id}) (\tilde \rho w_h),\nabla \tilde i_h \tilde u)\\
		& \le \ltwonorm{\tilde \rho^\half \nabla w_h} \ltwonorm{\tilde \rho^\half\nabla \tilde i_h \tilde u} + c \ltwonorm{\tilde \rho^{-\half}w_h} \ltwonorm{\tilde \rho^\half\nabla \tilde i_h \tilde u}\\
		& \quad + \ltwonorm{\tilde \rho^{-\half}\nabla(i_h-\operatorname{id})(\tilde \rho w_h)} \ltwonorm{\tilde \rho^\half\nabla \tilde i_h \tilde u}\\
		& \le \frac{1}{4} \ltwonorm{\tilde \rho^\half \nabla w_h}^2 + c \ltwonorm{\tilde \rho^{-\half}
			w_h}^2 + c\ltwonorm{\tilde \rho^\half\nabla \tilde i_h \tilde u}^2,
	\end{aligned}
	\]
	where we again used the super convergence estimate from \cref{FEM:lemma:superconvergence_ih}.
	Similar, we estimate $I_2$ as
	\[
	I_2 \le \frac{1}{4} \ltwonorm{\tilde \rho^\half \nabla w_h}^2 + c \ltwonorm{\tilde \rho^{-\half} w_h}^2.
	\]
	Putting terms together and absorbing terms involving $\ltwonorm{\tilde \rho^\half \nabla w_h}^2$ into the left-hand side we get
	\[
	\ltwonorm{\tilde \rho^\half \nabla w_h}^2 \le c \ltwonorm{\tilde \rho^{-\half} w_h}^2 + c\ltwonorm{\tilde \rho^\half\nabla \tilde i_h \tilde u}^2.
	\]
	The term $\ltwonorm{\tilde \rho^\half\nabla \tilde i_h \tilde u}$ is estimated in \eqref{FEM:eq_in_proof_harmonic_weighted:2}. It remains to estimate  $\ltwonorm{\tilde \rho^{-\half} w_h}$. We use $\tilde \rho \ge c h$ and obtain
	\[
	\ltwonorm{\tilde \rho^{-\half} w_h} \le c h^{-\half} \ltwonorm{w_h} \le  c h^{-\half} \ltwonorm{\tilde u - \tilde i_h \tilde u} + c h^{-\half} \ltwonorm{\tilde u - u_h}.
	\]
	For the interpolation error, we directly get
	\[
	c h^{-\half} \ltwonorm{\tilde u - \tilde i_h \tilde u} \le c h^{\half} \ltwonorm{\nabla \tilde u} \le
	c h^{\half} \norm{P_h^\partial q}_{H^{\nhalf}(\partial \Omega)} \le c \ltwonormdo{P_h^\partial q},
	\]
	where we used  again the $H^1$ estimate for $\tilde u$ and an
	inverse estimate. As $P_h^\partial (P_h^\partial q)=P_h^\partial q$, we obtain from \cref{theorem:state_error_est}
	\[
	c h^{-\half} \ltwonorm{\tilde u - u_h} \le c \ltwonormdo{P_h^\partial q}.
	\]
	Putting terms together yields
	\[
	\ltwonorm{\tilde \rho^{-\half} w_h} \le  c \ltwonormdo{P_h^\partial q}
	\]
	and therefore
	\[
	\ltwonorm{\tilde \rho^{\half} \nabla w_h} \le  c \ltwonormdo{P_h^\partial q}.
	\]
\end{enumerate}
Thus, we have proven
\[
	\ltwonorm{\tilde \rho^{\half} \nabla u_h}\le \ltwonorm{\tilde \rho^{\half} \nabla \tilde i_h \tilde u}+\ltwonorm{\tilde \rho^{\half} \nabla (u_h-\tilde i_h \tilde u)}\le c \ltwonormdo{P_h^\partial q}.
\]
\end{proof}

\section{Discretization of the optimal control problem}\label{sec5}
To discretize the optimal control problem \eqref{DirichletCon:eq:reduced_problem} we introduce the discrete reduced cost functional
\[
j_h \colon Q \to \R, \quad j_h(q) = J(q,S_h q),
\]
where $S_h \colon q \mapsto u_h(q)$ is the discrete solution operator defined by \eqref{FEM:eq:poisson_inhomogenDirichlet}. We discuss two concepts for control discretization: variational discretization and cellweise linear discretization. For the variational discretization we choose the discrete control space $Q_h$ as $Q_h = Q=L^2(\partial \Omega)$, cf. \cite{Hinze:2005} and \cite{DeckelnichGuentherHinze:2009} in the case of Dirichlet boundary control. For the cellweise linear discretization we choose $Q_h = V_h^\partial$. In both cases we define the discrete admissible set as $\Qadh = \Qad\cap Q_h$. This leads to the following discretized problem
\begin{equation}\label{DirichletCon:eq:reduced_problem_h}
	\text{minimize } j_h(q), \quad q \in \Qadh.
\end{equation}
As on the continuous level the functional $j_h$ is continuous and strictly convex. The existence and uniqueness of a discrete optimal control $\oq_h \in \Qadh$ follows by standard arguments for both choices of the control discretization. The corresponding discrete state $\ou_h = u_h(\oq_h)$ is referred as the discrete optimal state. The discrete reduced cost functional is two times continuously differentiable and the directional derivatives at $q\in Q$ in the direction $\dq \in Q$ are given as
\begin{equation}\label{eq:dj_h}
j'_h(q)(\dq) = (u_h-u_d,\du_h) + \alpha(q,\dq)_{\partial\Om}
\end{equation}
and
\begin{equation}\label{eq:ddj_h}
j''_h(q)(\dq,\dq) = \ltwonorm{\du_h}^2 + \alpha \ltwonormdo{\dq}^2,
\end{equation}
where $u_h = S_hq$ and $\du_h = S_h\dq$. We require also an adjoint representation of $j'_h(q)(\dq)$. Let $q\in Q$ and $u_h = S_h q$ be
the solution of the discrete state equation \eqref{FEM:eq:poisson_inhomogenDirichlet}. We define $z_h$ as
\begin{equation}\label{DirichletCon:eq:adjoint_weak_h}
	z_h \in V_h \quad:\quad (\nabla \phi_h,\nabla z_h) = (u_h-u_d,\phi_h) \quad \text{for all } \phi_h \in V_h.
\end{equation}
Similar to the discussion in \cref{DirichletCon:remark:dj_rep_domain} we obtain the following representation
\begin{equation}
	j'_h(q)(\dq) = (u_h-u_d,B_h P_h^\partial\dq)- (\nabla z_h,\nabla B_h P_h^\partial\dq)+ \alpha(q,\dq)_{\partial\Om},
\end{equation}
where $B_h \colon V_h^\partial \to \widehat V_h$ is an arbitrary discrete extension operator. The above representation suggests the definition of a so called \emph{discrete variational normal trace} of $z_h$ solving \cref{DirichletCon:eq:adjoint_weak_h}. We set
\begin{multline}\label{DirirchletCon:eq:var_partial_n_z_h}
	\partial_n^h z_h \in V_h^\partial \quad:\quad (\partial_n^h z_h,\psi_h)_{\partial \Omega} =
	(\nabla z_h,\nabla B_h \psi_h)\\-(u_h-u_d,B_h \psi_h) \quad \text{for all } \psi_h \in V_h^\partial,
\end{multline}
which leads to an equivalent representation
\[
j'_h(q)(\dq) = \alpha(q,\dq)_{\partial\Om} - (\partial_n^h z_h,P_h^\partial\dq)_{\partial\Om} = (\alpha q -\partial_n^h z_h,\dq)_{\partial\Om}.
\]
The last step holds due to the definition of the $L^2(\partial\Omega)$ projection. The following lemma provides a representation for $j'-j'_h$.
\begin{lemma}\label{DirichletCon:lemma:representation_dj-djh}
	Let $q,\dq\in Q$, let $u=u(q)$ be the solution of the continuous state equation and $u_h = u_h(q)$. Let $z$ be the corresponding continuous adjoint solving \eqref{DirichletCon:eq:adjoint_weak} and $z_h$ the discrete adjoint \eqref{DirichletCon:eq:adjoint_weak_h}.  There holds
	\begin{multline*}
		j'(q)(\dq)-j'_h(q)(\dq) = (\partial_n z - P_h^\partial \partial_n z,\dq)_{\partial \Omega}\\
		+ (u-u_h,S_h P_h^\partial \dq) - (\nabla(z-i_h z),\nabla S_h P_h^\partial \dq),
	\end{multline*}
	where $S_h \colon Q \to \widehat V_h$ is the discrete solution operator defined by \eqref{FEM:eq:poisson_inhomogenDirichlet} and $i_h$ is the Lagrange  interpolation operator.	
\end{lemma}
\begin{proof}
	There holds
	\begin{multline*}
		j'(q)(\dq)-j'_h(q)(\dq) = \left(j'(q)(\dq-P_h^\partial \dq)-j'_h(q)(\dq-P_h^\partial \dq)\right)\\ + \left(j'(q)(P_h^\partial \dq)-j'_h(q)(P_h^\partial \dq)\right) = I_1 + I_2
	\end{multline*}
	For $I_1$ we have by the representations of $j'$ and $j'_h$
	\[
	\begin{aligned}
		I_1 &= (\partial_n z - \partial_n^h z_h,\dq-P_h^\partial \dq)_{\partial \Omega}\\
		&=(\partial_n z - P_h^\partial \partial_n z,\dq-P_h^\partial \dq)_{\partial \Omega}\\
		&=(\partial_n z - P_h^\partial \partial_n z,\dq)_{\partial \Omega},
	\end{aligned}
	\]
	where we have used the definition of the $L^2(\partial \Omega)$ projection $P_h^\partial$. For $I_2$ we use the representation of $j'$ from \cref{DirichletCon:remark:dj_rep_domain} and of $j'_h$  leading to
	\[
	I_2 = (u-u_d,B P_h^\partial\dq) - (\nabla BP_h^\partial\dq, \nabla z) - (u_h-u_d,B_h P_h^\partial\dq) + (\nabla B_h P_h^\partial\dq, \nabla z_h)
	\]
	with an arbitrary extension operator $B$ and discrete extension operator $B_h$. We choose $B_h$ as discrete harmonic extension, i.e. equal to the discrete solution operator $B_h = S_h$. Since $B$ is applied here only to discrete functions we may choose $B=B_h$. This results in
	\[
	\begin{aligned}
		I_2 &= (u-u_h,S_h P_h^\partial\dq) - (\nabla S_h P_h^\partial\dq, \nabla(z- z_h))\\
		&= (u-u_h,S_h P_h^\partial\dq) - (\nabla S_h P_h^\partial\dq, \nabla(z- i_h z)),
	\end{aligned}
	\]
	where in the last step we used the definition of the discrete solution operator $S_h$ and the fact that $z_h, i_h z \in V_h$. This completes the proof.
\end{proof}

Next lemma provides an estimate for $j'(\oq)-j'_h(\oq)$.
\begin{lemma}\label{lemma:dj-djh}
Let $0<s<s_\Omega$ with $s_\Omega$ from \eqref{eq:s_Omega}. Let $u_d \in H^1(\Omega)$. Let $(\oq,\ou,\oz)$ be the solution of the optimality system \eqref{DirichletCon:eq:opt_sys}. Then for all $\dq \in L^2(\partial \Omega)$ there holds
\[
\abs{j'(\oq)(\dq)-j'_h(\oq)(\dq)} \le c h^{\half+s}\norm{u_d}_{H^1(\Omega)} \norm{\dq}_{L^2(\partial \Omega)}
\]
with a constant $c>0$ independent of $h$, $u_d$ and $\dq$.
\end{lemma}
\begin{proof}
	By the representation from \cref{DirichletCon:lemma:representation_dj-djh} we obtain
	\begin{multline*}
		j'(\oq)(\dq)-j'_h(\oq)(\dq) = (\partial_n \oz - P_h^\partial \partial_n \oz,\dq)_{\partial \Omega}\\
		+ (\ou-\tilde u_h,S_h P_h^\partial \dq) - (\nabla(\oz-i_h \oz),\nabla S_h P_h^\partial \dq),
	\end{multline*}
	where $\tilde u_h = S_h \oq$ is the discrete solution to the continuous optimal control.
	 By \cref{theorem:W22-s_oz} we have 
	\[
	\oz \in W^{2,2}_{-s}(\Omega) \quad\text{and}\quad \partial_n \oz, \oq \in H^{\half + s}(\partial \Omega)
	\]
	and the corresponding estimates
	\[
	\norm{\oz}_{W^{2,2}_{-s}(\Omega)} + \norm{\partial_n \oz}_{H^{\half + s}(\partial \Omega)} + \norm{\oq}_{H^{\half + s}(\partial \Omega)} \le c \norm{u_d}_{H^1(\Omega)}
	\]
	hold. The first term in the above representation of $j'(\oq)-j'_h(\oq)$ is directly estimated by standard estimates for the $L^2(\partial\Omega)$ projection as
	\[
	\begin{aligned}
		\abs{(\partial_n \oz - P_h^\partial \partial_n \oz,\dq)_{\partial \Omega}} &\le ch^{\half+s} \norm{\partial_n \oz}_{H^{\half + s}(\partial \Omega)} \ltwonormdo{\dq}\\
		& \le ch^{\half+s} \norm{u_d}_{H^1(\Omega)} \ltwonormdo{\dq}.
	\end{aligned}
	\]
	The term $(\ou-\tilde u_h,S_h P_h^\partial \dq)$ is estimated by \cref{theorem:state_error_est} and by the stability of $S_h$ as 
	\[
	\begin{aligned}
		\abs{(\ou-\tilde u_h,S_h P_h^\partial \dq)}  &\le \ltwonorm{\ou-\tilde u_h} \ltwonorm{S_h P_h^\partial \dq}\\
		&\le c h \norm{\oq}_{H^{\nhalf}(\partial \Omega)} \norm{\dq}_{L^2(\partial \Omega)}\\
		& \le c h \norm{u_d}_{H^1(\Omega)} \norm{\dq}_{L^2(\partial \Omega)}.
	\end{aligned}
	\]
	For the term involving $\oz-i_h \oz$ we use the weight $\tilde \rho$, see \eqref{eq:tilde_rho}, and obtain
	\[
	\begin{aligned}
		\abs{(\nabla(\oz-i_h \oz),\nabla S_h P_h^\partial \dq)} &= \abs{(\tilde \rho^{-\half}\nabla(\oz-i_h \oz),\tilde \rho^\half\nabla S_h P_h^\partial \dq)}\\
		& \le \ltwonorm{\tilde \rho^{-\half}\nabla(\oz-i_h \oz)} \ltwonorm{\tilde \rho^\half\nabla S_h P_h^\partial \dq}.
	\end{aligned}
	\]
	The weighted interpolation error is estimated by \cref{FEM:lemma:est_interpolation_general_weight}
	\[
	\begin{aligned}
		\ltwonorm{\tilde \rho^{-\half}\nabla(\oz-i_h \oz)} &\le ch \ltwonorm{\tilde \rho^{-\half}\nabla^2 \oz}\\
		&\le c h \linfnorm{\tilde \rho^{-\half+s}} \ltwonorm{\tilde \rho^{-s}\nabla^2 \oz}\\
		&\le c h (\kappa h)^{-\half+s} \ltwonorm{\rho^{-s}\nabla^2 \oz}\\
		&\le ch^{\half + s} \norm{\oz}_{W^{2,2}_{-s}(\Omega)}\\
		&\le  ch^{\half + s} \norm{u_d}_{H^1(\Omega)},
	\end{aligned}
	\]
	where we also used $\tilde \rho \ge \kappa h$, $s< \half$ and $\tilde \rho \ge \rho$.
	Using \cref{theorem:discrete_state_stab} and the stability of the $P_h^\partial$ in $L^2(\partial\Omega)$ we get
	\[
	\ltwonorm{\tilde \rho^\half\nabla S_h P_h^\partial \dq} \le c \ltwonormdo{P_h^\partial \dq} \le c \ltwonormdo{\dq}
	\]
	resulting in
	\[
	\abs{(\nabla(\oz-i_h \oz),\nabla S_h P_h^\partial \dq)} \le ch^{\half + s} \norm{u_d}_{H^1(\Omega)}  \ltwonormdo{\dq}.
	\]
	This completes the proof.
\end{proof}

Under the assumption on $\Omega$ that $\lambda_\Omega>\frac{3}{2}$ we obtain the following corollary.

\begin{corollary}\label{cor:dj-djh}
Let the critical exponent $\lambda_\Omega$ from \eqref{eq:lambda_Omega} fulfill $\lambda_\Omega>\frac{3}{2}$. Then there holds in the setting of \cref{lemma:dj-djh}
\[
\abs{j'(\oq)(\dq)-j'_h(\oq)(\dq)} \le c h\lh\norm{u_d}_{H^1(\Omega)} \norm{\dq}_{L^2(\partial \Omega)}.
\]
\end{corollary}
\begin{proof}
	We start again by the representation from \cref{DirichletCon:lemma:representation_dj-djh}
	\begin{multline*}
		j'(\oq)(\dq)-j'_h(\oq)(\dq) = (\partial_n \oz - P_h^\partial \partial_n \oz,\dq)_{\partial \Omega}\\
		+ (\ou-\tilde u_h,S_h P_h^\partial \dq) - (\nabla(\oz-i_h \oz),\nabla S_h P_h^\partial \dq).
	\end{multline*}
From \cref{DirichletCon:theorem:W22-s_oz_with_s_to_12} we have (since $\lambda_\Omega>\frac{3}{2}$)
\[
\norm{\oz}_{W^{2,2}_{-s}}(\Omega) \le \frac{c}{1-2s} \norm{u_d}_{H^1(\Omega)}
\]
for all $s< \half$ as well as
\[
\norm{\partial_n \oz}_{H^1(\partial \Omega)} +\norm{\oq}_{H^1(\partial \Omega)} \le  c \norm{u_d}_{H^1(\Omega)}.
\]
Thus, the first term in the above representation can be estimated as
\[
(\partial_n \oz - P_h^\partial \partial_n \oz,\dq)_{\partial \Omega} \le c h \norm{\dq}_{L^2(\partial \Omega)} \norm{\partial_n \oz}_{H^1(\partial \Omega)} \le ch \norm{u_d}_{H^1(\Omega)} \norm{\dq}_{L^2(\partial \Omega)}.
\]
The second term is estimated as in the previous theorem leading to
\[
\abs{(\ou-\tilde u_h,S_h P_h^\partial \dq)} \le ch \norm{u_d}_{H^1(\Omega)} \norm{\dq}_{L^2(\partial \Omega)}.
\]
For the third term we get following the arguments in the proof of the previous theorem
\[
	\abs{(\nabla(\oz-i_h \oz),\nabla S_h P_h^\partial \dq)} \le \frac{c}{1-2s}h^{\half + s} \norm{u_d}_{H^1(\Omega)}  \ltwonormdo{\dq}
\] 
for every $s< \half$. The choice $s = \half - \lh^{-1}$ provides the desired estimate by observing that
\[
	1-2s= \abs{\ln h}^{-1}\quad\text{and}\quad h^{-\abs{\ln h}^{-1}}= h^{(\ln h)^{-1}}= \mathrm{e}^{\ln h (\ln h)^{-1} } = \mathrm{e}.
\]
\end{proof}

\subsection{Variational discretization}
In this section we consider the choice $Q_h=Q=L^2(\partial \Omega)$ and correspondingly $\Qadh=\Qad$ in \eqref{DirichletCon:eq:reduced_problem_h}. In this case the discrete optimal control $\oq_h$ is characterized by
\[
\oq_h \in \Qad \quad : \quad j'_h(\oq_h)(\dq-\oq_h) \ge 0 \quad \text{for all } \dq \in \Qad.
\]
The next theorem provides an error estimate for this choice of discretization concept.
\begin{theorem}\label{theorem:q_error_est_var}
Let $\oq$ be the solution of \eqref{DirichletCon:eq:reduced_problem} and $\oq_h$ the solution of \eqref{DirichletCon:eq:reduced_problem_h} for the choice $Q_h=Q$. Let $0<s<s_\Omega$ with $s_\Omega$ from \eqref{eq:s_Omega} and $u_d \in H^1(\Omega)$. Then, there holds
\[
\ltwonormdo{\oq-\oq_h} \le c h^{\half+s}\norm{u_d}_{H^1(\Omega)}
\]
with a constant $c>0$ independent of $h$ and $u_d$. If $\Omega$ fulfills $\lambda_\Omega > \frac{3}{2}$ then there holds
\[
\ltwonormdo{\oq-\oq_h} \le c h\lh\norm{u_d}_{H^1(\Omega)}.
\]
\end{theorem}
\begin{proof}
From the representation \eqref{eq:ddj_h} of $j''_h$ we obtain
\[
\alpha \ltwonormdo{\oq-\oq_h}^2 \le j''_h(\oq)(\oq-\oq_h,\oq-\oq_h) = j'_h(\oq)(\oq-\oq_h) - j'_h(\oq_h)(\oq-\oq_h)
\]
by affine linearity of $j'_h(q)$ with respect to $q$. By the fact that $\oq_h$ is admissible for the continuous problem and $\oq$ is admissible for the discrete we obtain by both (continuous and discrete) optimality conditions
\[
- j'_h(\oq_h)(\oq-\oq_h) \le 0 \le - j'(\oq)(\oq-\oq_h).
\]
This results in
\[
\alpha \ltwonormdo{\oq-\oq_h}^2 \le j'_h(\oq)(\oq-\oq_h) - j'(\oq)(\oq-\oq_h).
\]
An application of \cref{lemma:dj-djh} and \cref{cor:dj-djh} leads to the desired estimates.
\end{proof}

\subsection{Cellwise linear discretization}
In this section we consider the choice $Q_h=V_h^\partial$ and correspondingly $\Qadh=\Qad\cap V_h^\partial$ in \eqref{DirichletCon:eq:reduced_problem_h}. In this case the discrete optimal control $\oq_h$ is characterized by
\[
\oq_h \in \Qadh \quad : \quad j'_h(\oq_h)(\dq_h-\oq_h) \ge 0 \quad \text{for all } \dq_h \in \Qadh.
\]
To proceed we require an interpolation operator on $Q$, which preserves admissibility and have improved approximation property in some negative norms. The $L^2(\partial \Omega)$ projection on $V_h^\partial$ can not be used, since it does not preserve the admissibility of functions from $\Qad$. We consider the quasi-interpolation $\pi_h \colon L^1(\partial \Omega) \to V_h^\partial$ as introduced in \cite{Carstensen:1999}. In \cite{Carstensen:1999} this operator is defined on the domain and we use it on the boundary, but the definition (and the properties) are very similar. We consider the set ${\cal N}_h^\partial$ of all nodes of the the mesh $\T_h$ on the boundary $\partial \Omega$. For every node $y \in {\cal N}_h^\partial$ let $\phi_y\in V_h^\partial$ be the local (nodal) basis function (hat-function) on the boundary. The operator $\pi_h$ is defined as
\[
\pi_h q = \sum_{y \in {\cal N}_h^\partial} \frac{(q,\phi_y)_{\partial \Omega}}{(1,\phi_y)_{\partial \Omega}} \phi_y.
\]
The definition directly implies that for every $q \in \Qad$ we have $\pi_h q \in \Qadh$ (note that $q_a$ and $q_b$ are constants).
\begin{lemma}\label{lemma:est_pi_h}
Let $q\in H^t(\partial \Omega)$ for some $0\le t \le 1$. Then there holds
\[
\ltwonormdo{q-\pi_h q} \le c h^t \norm{q}_{H^t(\partial \Omega)}
\]
and
\[
\norm{q-\pi_h q}_{H^{-t}(\partial \Omega)} \le c h^{2t} \norm{q}_{H^t(\partial \Omega)}
\]
with a constant $c>0$ independent of $h$ and $q$.
\end{lemma}
\begin{proof}
For the case $t=1$ the proof of the $L^2(\partial\Omega)$ estimate follows as in \cite[Theorem 3.1]{Carstensen:1999}. The negative norm estimate is shown as in \cite[Lemma 4.5]{ReyesMeyerVexler:2008}. The general case of the $L^2(\partial\Omega)$ estimate is then obtained by interpolation. See also \cite[Lemma 2.14]{ApelNicaisePfefferer:2016} for more details.
\end{proof}

The next theorem provides an error estimate for this choice of discretization concept.
\begin{theorem}\label{theorem:q_error_est_p1}
	Let $\oq$ be the solution of \eqref{DirichletCon:eq:reduced_problem} and $\oq_h$ the solution of \eqref{DirichletCon:eq:reduced_problem_h} for the choice $Q_h=V_h^\partial$. Let $0<s<s_\Omega$ with $s_\Omega$ from \eqref{eq:s_Omega} and $u_d \in H^1(\Omega)$. Then, there holds
	\[
	\ltwonormdo{\oq-\oq_h} \le c h^{\half+s}\norm{u_d}_{H^1(\Omega)}
	\]
	with a constant $c>0$ independent of $h$ and $u_d$. If $\Omega$ fulfills $\lambda_\Omega > \frac{3}{2}$ then there holds
	\[
	\ltwonormdo{\oq-\oq_h} \le c h\lh\norm{u_d}_{H^1(\Omega)}.
	\]
\end{theorem}
\begin{proof}
We split the error as
\[
\ltwonormdo{\oq-\oq_h} \le \ltwonormdo{\oq-\pi_h \oq} + \ltwonormdo{\pi_h \oq-\oq_h}.
\]
We start with considering the general case $0<s<s_\Omega$ and comment on the second case later on. The estimate for $\oq-\pi_h \oq$ follows from \cref{lemma:est_pi_h} using the regularity $\oq \in H^{\nhalf+s}(\partial \Omega)$ from \cref{theorem:W22-s_oz} leading to
\[
\ltwonormdo{\oq-\pi_h \oq} \le c h^{\half+s} \norm{\oq}_{H^{\nhalf+s}(\partial \Omega)} \le c h^{\half+s} \norm{u_d}_{H^1(\Omega)}.
\]
As in the proof of \cref{theorem:q_error_est_var} we use the coercivity of $j''_h$ to get
\[
\alpha \ltwonormdo{\pi_h \oq-\oq_h}^2 \le j''_h(\oq)(\pi_h \oq-\oq_h,\pi_h\oq-\oq_h) = j'_h(\pi_h \oq)(\pi_h \oq-\oq_h) - j'_h(\oq_h)(\pi_h\oq-\oq_h).
\]
Using the fact that $\pi_h \oq \in \Qadh$ and $\oq_h \in \Qad$ we can exploit  both (continuous and discrete) optimality conditions
\[
- j'_h(\oq_h)(\pi_h \oq-\oq_h) \le 0 \le - j'(\oq)(\oq-\oq_h).
\]
This results in
\[
\begin{aligned}
\alpha \ltwonormdo{\pi_h \oq-\oq_h}^2 &\le j'_h(\pi_h \oq)(\pi_h \oq-\oq_h) - j'(\oq)(\oq-\oq_h)\\
& = I_1 + I_2 + I_3
\end{aligned}
\]
with
\[
I_1 = j'_h(\pi_h \oq)(\pi_h \oq-\oq_h) - j'_h(\oq)(\pi_h \oq-\oq_h),
\]
\[
I_2 = j'_h(\oq)(\pi_h \oq-\oq_h) - j'(\oq)(\pi_h \oq-\oq_h)
\]
and
\[
I_3 = j'(\oq)(\pi_h \oq-\oq_h)-j'(\oq)(\oq-\oq_h) = j'(\oq)(\pi_h \oq - \oq).
\]
For $I_1$ we get using the primal representation of $j'_h$ from \eqref{eq:dj_h}
\[
I_1 = (S_h(\pi_h \oq-\oq),S_h(\pi_h \oq-\oq_h)) + \alpha(\pi_h \oq-\oq,\pi_h \oq-\oq_h)_{\partial \Omega},
\]
which results by the stability of $S_h:L^2(\partial\Omega)\to L^2(\Omega)$ and Young's inequality in
\[
\abs{I_1} \le c \ltwonormdo{\pi_h \oq-\oq} \ltwonormdo{\pi_h \oq-\oq_h} \le \frac{\alpha}{4} \ltwonormdo{\pi_h \oq-\oq_h}^2 + c \ltwonormdo{\pi_h \oq-\oq}^2.
\]
The term $I_2$ is estimated by \cref{lemma:dj-djh} employing the regularity results from \cref{theorem:W22-s_oz} leading to
\[
\abs{I_2} \le c h^{\half+s} \norm{u_d}_{H^1(\Omega)} \ltwonorm{\pi_h \oq-\oq_h}
\le \frac{\alpha}{4} \ltwonormdo{\pi_h \oq-\oq_h}^2 + c h^{1+2s} \norm{u_d}_{H^1(\Omega)}^2.
\]
For $I_3$ we use the negative norm estimate from \cref{lemma:est_pi_h} and again the regularity results from \cref{theorem:W22-s_oz} leading to
\[
\begin{aligned}
\abs{I_3} &= \abs{(\alpha \oq - \partial_n \oz,\pi_h \oq - \oq)_{\partial \Omega}} \le c h^{1+2s} \norm{\alpha \oq - \partial_n \oz}_{H^{\nhalf+s}(\partial \Omega)} \norm{\oq}_{H^{\nhalf+s}(\partial \Omega)}\\
&\le  c h^{1+2s} \norm{u_d}_{H^1(\Omega)}^2.
\end{aligned}
\]
Putting terms together and absorbing terms in the left-hand side we complete the proof for the general case. The estimate for the second case $\lambda_\Omega > \frac{3}{2}$ follows similarly by using \cref{cor:dj-djh} instead of \cref{lemma:dj-djh} and by using the regularity results from \cref{DirichletCon:theorem:W22-s_oz_with_s_to_12} instead of \cref{theorem:W22-s_oz}.
\end{proof}

\section{Numerical results}\label{sec6}
In this section we provide numerical examples to verify the order of convergence. We  consider
the optimal control problem \eqref{DirichletCon:eq:problem} with the slightly
modified state equation
\[
		\begin{aligned}
	-\Lap u &= f &\quad&\text{in } \Omega,\\
	u &=q &\quad&\text{on } \partial \Omega
\end{aligned}
\]
in the domain $\Om = \Om_\omega\subset\R^3$ given for some $\omega = \omega_\Omega\in[\frac\pi2,\pi)$ as
\[
\Om_\omega=\left((-1,1)^2\cap \Set{(r\cos\varphi,r\sin\varphi)^T |
	r\in(0,\infty),~\varphi\in(0,\omega)}\right)\times(0,1),
\]
see \cref{DirichletCon:fig:Omega3D}. Here, we use cylindrical coordinates $(r,\varphi,x_3)$ of $x\in\R^3$, where $r$ and
$\varphi$ are given as above.  With $\lambda=\lambda_\Omega=\frac\pi\omega$, the optimal state and adjoint state
are chosen as
\[
\begin{aligned}
	\ou(x)&=-\lambda
	r(x)^{\lambda-1}(1-x_1^2)(1-x_2^2)x_3^2(1-x_3)^2\\
	&\quad+2r(x)^\lambda\sin(\lambda\varphi(x))(x_1^2+x_2^2-2)x_3^2(1-x_3)^2,\\
	\oz(x)&=r(x)^\lambda \sin(\lambda\phi(x))(1-x_1^2)(1-x_2^2)x_3^2(1-x_3)^2.
\end{aligned}
\]
The optimal control is chosen as restriction $\oq = \ou\bigr\rvert_{\partial\Om}$. The control bounds $q_a<0<q_b$ are chosen here such that they do not become active. The adjoint state $\oz$ fulfills homogeneous Dirichlet boundary conditions on $\partial \Om_\omega$ and a direct
calculation shows that the optimality system \cref{DirichletCon:eq:opt_sys} is
fulfilled for $\alpha=1$. The right-hand side $f$ and the
desired state $u_d$ are calculated by means of $\ou$ and $\oz$ as $f=-\Lap\ou$ and $u_d=\ou+\Lap\oz$.

\tdplotsetmaincoords{70}{-10}

\begin{figure}
	\centering
	\begin{tikzpicture}[scale=1.9, tdplot_main_coords]
		\draw (0,0,1) -- (1,0,1) -- (1,1,1) -- (0,1,1) -- cycle;
		\draw (0,0,0) -- (1,0,0);
		\draw[dashed] (1,0,0) -- (1,1,0) -- (0,1,0);
		\draw (0,1,0) -- (0,0,0);
		\draw (0,0,0) -- (0,0,1);
		\draw (1,0,0) -- (1,0,1);
		\draw[dashed] (1,1,0) -- (1,1,1);
		\draw (0,1,0) -- (0,1,1);
		\node at (0.5,0.5,0.5) {$\Om_{\frac\pi2}$};
	\end{tikzpicture}\qquad
	\begin{tikzpicture}[scale=1.9, tdplot_main_coords]
		\draw (0,0,1) -- (1,0,1) -- (1,1,1) -- (-0.57735026919,1,1) -- cycle;
		\draw (0,0,0) -- (1,0,0);
		\draw[dashed] (1,0,0) -- (1,1,0) -- (-0.57735026919,1,0);
		\draw (-0.57735026919,1,0) -- (0,0,0);
		\draw (0,0,0) -- (0,0,1);
		\draw (1,0,0) -- (1,0,1);
		\draw[dashed] (1,1,0) -- (1,1,1);
		\draw (-0.57735026919,1,0) -- (-0.57735026919,1,1);
		\node at (0.3558,0.5,0.5) {$\Om_{\frac{2\pi}3}$};
	\end{tikzpicture}\qquad
	\begin{tikzpicture}[scale=1.9, tdplot_main_coords]
		\draw (0,0,1) -- (1,0,1) -- (1,1,1) -- (-1,1,1) -- cycle;
		\draw (0,0,0) -- (1,0,0);
		\draw[dashed] (1,0,0) -- (1,1,0) -- (-1,1,0);
		\draw (-1,1,0) -- (0,0,0);
		\draw (0,0,0) -- (0,0,1);
		\draw (1,0,0) -- (1,0,1);
		\draw[dashed] (1,1,0) -- (1,1,1);
		\draw (-1,1,0) -- (-1,1,1);
		\node at (0.25,0.5,0.5) {$\Om_{\frac{3\pi}4}$};
	\end{tikzpicture}
	\caption{Domain $\Om_\omega$ for
		$\omega\in\set{\frac\pi2,\frac{2\pi}3,\frac{3\pi}4}$}\label{DirichletCon:fig:Omega3D}
\end{figure}
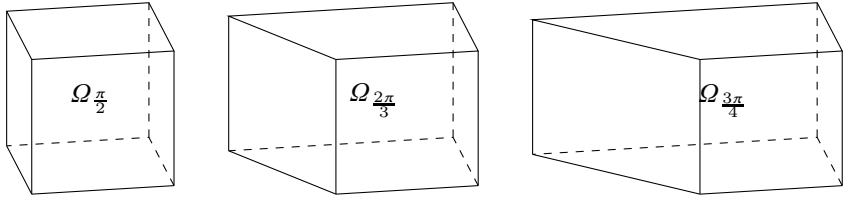

\begin{figure}
	\begin{tikzpicture}
		\begin{loglogaxis} [%
			xlabel={mesh size $h$},
			grid=major,
			legend pos=south east,
			legend cell align=left,
			width=\textwidth,
			height=0.6\textwidth,
			cycle list name=black white,
			]
			
			\addplot table [x=h, y=eq, col sep=semicolon] {errors3D_12_dist.csv};
			\addlegendentry{$\omega=\frac\pi2$}
			\addplot table [x=h, y=eq, col sep=semicolon] {errors3D_23.csv};
			\addlegendentry{$\omega=\frac{2\pi}3$}
			\addplot table [x=h, y=eq, col sep=semicolon] {errors3D_34.csv};
			\addlegendentry{$\omega=\frac{3\pi}4$}
			
			\logLogSlopeTriangle{0.25}{0.15}{0.26}{5/6}{black}{h^{\frac56}}
			\logLogSlopeTriangle{0.25}{0.15}{0.073}{1}{black}{h}
		\end{loglogaxis}
	\end{tikzpicture}
	\caption{Error $\ltwonormdo{\oq-\oq_h}$}\label{DirichletCon:fig:qerror3D}
\end{figure}

In \cref{DirichletCon:fig:qerror3D}, we present the development of the error
$\ltwonormdo{\oq-\oq_h}$ for $\omega\in\set{\frac\pi2,\frac{2\pi}3,\frac{3\pi}4}$ and $h$ tending to
zero. We exactly observe the orders of convergence as predicted by \cref{theorem:q_error_est_var} and \cref{theorem:q_error_est_p1}. Note that for the case with inactive control constraints both discretization concepts described above coincide.
We also note that the calculations for $\omega=\frac\pi2$ were performed on
disturbed meshes to avoid superconvergence effects.

\section{Acknowledgment}\label{sec7} We would like to thank Dr. Dominik Meidner for scientific exchange and for helping us to prepare the numerical examples.

\backmatter

\bibliography{lit}

\end{document}